\documentclass[onefignum,onetabnum]{siamart190516}

\usepackage{lipsum,bm}
\usepackage{amsfonts,wasysym}
\usepackage{graphicx}
\usepackage{epstopdf}
\usepackage{algorithmic}
\usepackage{amsopn}
\ifpdf
\DeclareGraphicsExtensions{.eps,.pdf,.png,.jpg}
\else
\DeclareGraphicsExtensions{.eps}
\fi

\newsiamremark{rem}{Remark}
\newsiamremark{defn}{Definition}
\newsiamthm{cor}{Corollary}
\newsiamthm{assum}{Assumption}
\newsiamthm{thm}{Theorem}
\newsiamthm{prop}{Proposition}
\newsiamthm{lem}{Lemma}

\newcommand{\norm}[1]{\left\Vert#1\right\Vert}
\newcommand{\abs}[1]{\left\vert#1\right\vert}
\newcommand{\set}[1]{\left\{#1\right\}}

\newcommand{\E}[1]{\mathbb{E}\left[#1\right]}

\newcommand{\diag}[1]{\mathrm{diag}(#1)}
\newcommand{\var}[1]{\mathrm{Var}\left(#1\right)}

\def\is{\mathrm{IS}}

\headers{On the error rate of IS with RQMC}{Z. He, Z. Zheng, and X. Wang}

\title{On the error rate of importance sampling with randomized quasi-Monte Carlo\thanks{Submitted to the editors DATE.
		\funding{This work of the first author was funded by the National Science Foundation of China (No. 12071154), Guangdong Basic and Applied Basic Research Foundation (No. 2021A1515010275), Guangzhou Science and Technology Program (No. 202102020407). And the third author was funded by the National Science Foundation of China (No. 720711119).}}}

\author{Zhijian He\thanks{Corresponding author. School of Mathematics, South China University of Technology, Guangzhou 510641, People's Republic of China (\email{hezhijian@scut.edu.cn}).}
	\and Zhan Zheng\thanks{Department of Mathematical Sciences, Tsinghua University, Beijing 100084, People's Republic of China (\email{zhengz15@mails.tsinghua.edu.cn}).}
	\and Xiaoqun Wang\thanks{Department of Mathematical Sciences, Tsinghua University, Beijing 100084, People's Republic of China (\email{wangxiaoqun@mail.tsinghua.edu.cn}).}}

\ifpdf
\hypersetup{
	pdftitle={On the error rate of importance sampling with randomized quasi-Monte Carlo},
	pdfauthor={Zhijian He, Zhan Zheng, and Xiaoqun Wang}
}
\fi

\begin{document}
	
	\maketitle
	
	\begin{abstract}
		Importance sampling (IS) is valuable in reducing the variance of Monte Carlo sampling for many areas, including finance, rare event simulation, and Bayesian inference. It is natural and obvious to combine quasi-Monte Carlo (QMC) methods with IS to achieve a faster rate of convergence. However, a naive replacement of Monte Carlo with QMC may not work well. This paper investigates the convergence rates of randomized QMC-based IS for estimating integrals with respect to a Gaussian measure, in which the IS measure is a Gaussian or $t$ distribution. We prove that if the target function satisfies the so-called boundary growth condition and the covariance matrix of the IS density has eigenvalues no smaller than 1, then randomized QMC with the Gaussian proposal has a root mean squared error of $O(N^{-1+\epsilon})$ for arbitrarily small $\epsilon>0$. Similar results of $t$ distribution as the proposal are also established. These sufficient conditions help to assess the effectiveness of IS in QMC. For some particular applications, we find that the Laplace IS, a very general approach to approximate the target function by a quadratic Taylor approximation around its mode, has eigenvalues smaller than 1, making the resulting integrand less favorable for QMC. From this point of view, when using Gaussian distributions as the IS proposal, a change of measure via Laplace IS may transform a favorable integrand into unfavorable one for QMC although the variance of Monte Carlo sampling is reduced. We also study the effect of positivization trick on the error rate when the integrand has mixed sign. If the smooth positivization proposed by Owen and Zhou (2000) is used, the rate $O(N^{-1+\epsilon})$ is retained. This is not the case if taking the positive and negative parts of the integrand. We also give some examples to verify our propositions and warn against naive replacement of MC with QMC under IS proposals. Numerical results suggest that using Laplace IS with $t$ distributions is more robust than that with Gaussian distributions.
	\end{abstract}
	
	
	
	\section{Introduction}
	Many problems in finance and statistics such as financial derivative pricing and Bayesian computation, can be reduced to the computations of expectations. Most of the expectations are integrals with respect to a Gaussian measure. For example, the underlying assets are usually driven by Brownian motions in security pricing, and the prior is typically assumed to be Gaussian for Bayesian computation. In this paper, we focus on the problem of estimating integrals with respect to a Gaussian measure. Monte Carlo (MC) methods become more often than ever the only computational feasible means. MC methods have many virtues, the most prominent is that the smoothness of the integrand is not needed (except that the integrand is square integrable) and that their convergence rate is dimension-independent. However, with the sample size of $N$, the crude MC has a convergence rate of $O(N^{-1/2})$ which may be too slow for practical applications. This paper aims at speeding up the MC methods by incorporating importance sampling (IS) and quasi-Monte Carlo (QMC) methods. 
	
	IS is a popular variance reduction method in the MC literature. It has the capacity to produce orders of magnitude for variance reduction, but it may also result in an estimate
	with infinite variance if it is not properly used. Glasserman et al. \cite{Glasserman1999b} studied IS for security pricing. Furthermore, IS has the advantage of handling rare events \cite{rubino2009rare}. More research find out that IS is more than just a variance reduction method. It can be used to study one distribution while sampling from another. We refer to Chapter 9 of \cite{owen2013monte} for a comprehensive review on IS.
	 
	In this paper, we investigate two common choices of importance densities, namely, the optimal drift importance sampling (ODIS) and the Laplace importance sampling (LapIS). The ODIS uses a multivariate normal density $N(\bm \mu_\star,\bm \Sigma)$ as the IS density ($\bm \Sigma$ is the original covariance matrix), and the LapIS uses a general multivariate normal density $N(\bm \mu_\star,\bm \Sigma_\star)$ as the IS density, where the mean and the covariance matrix are chosen to match the mode and the curvature of the integrand. The LapIS has been applied in statistics \cite{booth1999maximizing,Kuk1999Laplace,schi:2020}. We also replace the  multivariate normal distribution $N(\bm\mu,\bm \Sigma)$ with the multivariate $t$ distribution $t_{\nu}(\bm \mu,\bm \Sigma)$ as the IS proposal, where $\nu$ is the degree of freedom.
	
	QMC methods are deterministic versions of MC methods, which are based on low-discrepancy points or quasi-random numbers. In the last three decades, QMC methods are widely used in finance and statistics. For a $d$-dimensional integral, QMC quadrature rule yields a deterministic error bound $O(N^{-1}(\log N)^d)$ for certain regular  functions \cite{Niederreiter1992}, which is asymptotically superior to the canonical MC rate $O(N^{-1/2})$. In practice, we often use randomized QMC (RQMC), which not only keeps the convergence rate of QMC but also makes the error estimation possible. Different kinds of RQMC quadratures were proposed in the literature, such as random shifted lattice rules (see, e.g., \cite{sloan:joe:1994,LEcuyer2000}) and scrambled digital nets (see, e.g.,  \cite{owen1997b,owen1997a}). This work  is based on scrambled digital nets. We refer to \cite{dick:2010,lecu:2002} for details on QMC and RQMC. 
	
	In this paper, we study the effect of using IS in QMC. It is more challenging to derive effective variance reduction methods starting from a QMC aspect than from the MC aspect, since some properties of the integrand (such as the smoothness and the effective dimension) which are irrelevant in MC can be crucial in QMC. Particularly, \cite{Hickernell2005} found that the control variate has different effects in MC and QMC. What is the key difference of developing IS procedures in MC and QMC settings? Would using IS accelerate or reduce the rate of convergence in QMC? Dick et al. \cite{dick:2019} provided a weighted discrepancy bound of QMC-based IS and thus obtained an explicit error bound for sufficiently regular integrands. From the perspective of numerical experiments, \cite{Kuo2008} found that QMC outperformed MC for $d=25$ with log-likelihood integrals, and \cite{Zhang2021} found that neither ODIS nor LapIS is dominant, and the effect of using IS depends on the problem. This motivates us to study the error rate of QMC/RQMC when using IS and look at what factors have an impact on the convergence rate. To this end, a theoretical analysis of the impact of IS density on the regularity of the integrand is presented in this paper. We find that IS can bring enormous gains with a root mean squared error (RMSE) rate of nearly $O(N^{-1})$, beating MC significantly. It can also backfire, yielding an estimate with a larger variance than plain MC. 
	
	This paper is organized as follows. We review briefly QMC and RQMC methods, IS, ODIS and LapIS in Section~\ref{sec:pre}. In Section~\ref{sec:main}, we present sufficient conditions for assessing the RMSE rate of the RQMC-based IS estimator. In Section~\ref{sec:mult}, we consider the multivariate $t$ distribution as the IS proposal. In Section~\ref{sec:exam}, some examples are provided to show that inappropriate choice of IS density may backfire. Particularly, we compare the performance of ODIS and LapIS for estimating posterior expectations under Bayesian Logistic regression model. Section~\ref{sec:conc} concludes this paper. A lengthy proof of the main result is deferred to Appendix.
	
	\section{Preliminaries}\label{sec:pre}
	Consider the problem of estimating an integral
	\begin{equation}
		\label{eq:target}
		C = \int_{\mathbb{R}^d} G(\bm z)p(\bm z;\bm \mu_0,\bm \Sigma_0)d\bm z,
	\end{equation}
	where $p(\bm z;\bm \mu,\bm \Sigma)$ denotes the probability density function of $d$-dimensional normal distribution with the mean $\bm \mu$ and the covariance matrix $\bm \Sigma$. In this paper, we perform IS with multivariate normal or $t$ distribution as the proposal. We start with normal and leave $t$ distribution in Section~\ref{sec:mult}. Let $\bm L$ be the square root of the covariance matrix $\bm \Sigma$ satisfying $\bm L\bm L^T=\bm \Sigma$. Let $\bm I_d$ denote the square matrix of order $d$. By a change of measure and a change of variables, the integral \eqref{eq:target} can be changed to
	\begin{eqnarray}
		C &=& \int_{\mathbb{R}^d} G(\bm z) \frac{p(\bm z;\bm \mu_0,\bm \Sigma_0)}{p(\bm z;\bm \mu,\bm \Sigma)} p(\bm z;\bm \mu,\bm \Sigma)d\bm z \nonumber\\ 
		&=& \int_{\mathbb{R}^d} G(\bm \mu+\bm L\bm z) \frac{p(\bm \mu+\bm L\bm z;\bm \mu_0,\bm \Sigma_0)}{p(\bm \mu+\bm L\bm z;\bm \mu,\bm \Sigma)}p(\bm z;\bm 0,\bm I_d)d\bm z\notag\\
		&=& \int_{\mathbb{R}^d} G(\bm \mu+\bm L\bm z) W(\bm z)p(\bm z;\bm 0,\bm I_d)d\bm z\notag\\
		&=& \int_{\mathbb{R}^d} G_{\is}(\bm z) p(\bm z;\bm 0,\bm I_d) d\bm z\notag\\
		&:=& I(G_{\is}),\label{eq:ISform}
	\end{eqnarray}
	where the likelihood ratio (LR) function
	\begin{align}	
		W(\bm z)&:=W(\bm z;\bm\mu,\bm \Sigma) = \frac{p(\bm \mu+\bm L\bm z;\bm \mu_0,\bm \Sigma_0)}{p(\bm \mu+\bm L\bm z;\bm \mu,\bm \Sigma)} \notag\\&= \frac{\det (\bm \Sigma)^{1/2}}{\det (\bm \Sigma_0)^{1/2}} \exp\left\lbrace\frac{1}{2}\bm z^T \bm z - \frac{1}{2}(\bm \mu-\bm \mu_0 +\bm L\bm z)^T \bm \Sigma_0^{-1} (\bm \mu-\bm \mu_0 +\bm L\bm z)\right\rbrace,\label{eq:weightf}
	\end{align}
	and $G_{\is}(\bm z) = G(\bm \mu + \bm L \bm z)W(\bm z)$. If we do not use IS, it suffices to take $\bm\mu =\bm\mu_0$ and $\bm \Sigma=\bm \Sigma_0$.
	
	\subsection{Quasi-Monte Carlo methods}
	A crude MC estimate for \eqref{eq:ISform} is
	\begin{equation}
		\label{eq:crude}
		\hat{I}_N(G_{\is}) = \frac 1 N \sum_{i=1}^N  G_{\is}(\bm z_i),
	\end{equation}
	where $\bm z_i$ are independent and identically distributed (i.i.d.) standard normals. The RMSE of the crude MC is
	$$
		\sqrt{\mathbb{E}[(\hat{I}_N(G_{\is})-I(G_{\is}))^2]}=\frac {\sigma}{\sqrt{N}} ,
	$$
	where $\sigma^2$ is the variance of the integrand $G_{\is}(\bm z)$, defined as
	\begin{equation}\label{eq:MCvariance}
		\sigma^2:=\sigma^2(\bm \mu,\bm \Sigma) = \int_{\mathbb{R}^d} (G_{\is}(\bm z) -I(G_{\is}))^2p(\bm z;\bm 0,\bm I_d) d\bm z.
	\end{equation}
	Obviously, MC has an RMSE rate $O(N^{-1/2})$ if the integrand is square integrable. 
	
	To accelerate the rate of convergence, one may use QMC quadrature rule instead. QMC uses low-discrepancy points in the unit cube $[0,1)^d$. More specifically, QMC quadrature rule sets
	\begin{equation}
		\label{eq:crudeQMC}
		\hat{I}_N(G_{\is}) = \frac 1 N \sum_{i=1}^N  G_{\is}(\Phi^{-1}(\bm u_i)),
	\end{equation}   
	where  $\Phi(\cdot)$ is the cumulative distribution function (CDF) of the standard normal distribution, $\Phi^{-1}(\cdot)$ is its inverse (applied componentwise), and $\{\bm u_1,\dots,\bm u_N\}:=\mathcal{P}$ is a low-discrepancy point set in $[0,1)^d$. There are two main strategies for constructing low-discrepancy point sets: digital nets and lattice
	approaches. We refer to \cite{Niederreiter1992,dick:2010} for various constructions of such points. If $\bm u_1,\dots,\bm u_N$ are i.i.d. samples from uniform distribution $U(0,1)^d$, we arrive at the MC estimate \eqref{eq:crude} by taking $\bm z_i=\Phi^{-1}(\bm u_i)$. The QMC error bound is given by the well-known Koksma-Hlawka inequality \cite{Niederreiter1992}
	\begin{equation}\label{eq:hk}
		|\hat{I}_N(G_{\is})-I(G_{\is})| \leq D^*(\mathcal{P}) V_\mathrm{HK}(G_{\is}\circ \Phi^{-1}),
	\end{equation}
	where $D^*(\mathcal{P})$ is the star discrepancy of the point set $\mathcal{P}$ and $V_\mathrm{HK}(\cdot)$ is the variation (in the sense of Hardy and Krause) of a function defined over the unit cube. Several digital sequences achieve a star discrepancy $O(N^{-1}(\log N)^d)$. Therefore, if the variation is bounded, QMC integration has a deterministic error bound of $O(N^{-1}(\log N)^d)$, which is asymptotically superior to that of MC for a fixed dimension $d$.
	
	To facilitate the error estimation, one usually uses RQMC in which $\bm u_i$ are randomized suitably while keeping the low-discrepancy property. Among various RQMC methods, the scrambling technique proposed by \cite{Owen1995} gains its popularity in randomizing digital nets and sequences. Scrambled net quadrature has an RMSE of $o(N^{-1/2})$ for any squared integrable integrand, and has a faster rate $O(N^{-3/2}(\log N)^{(d-1)/2})$ for smooth integrands \cite{owen1997a,owen:2008}. We should note that the integrand $G_{\is}\circ \Phi^{-1}$ may have singularities along the surface of the unit cube due to the mapping $\Phi^{-1}(\cdot)$, resulting in an unbounded integrand. Owen \cite{owen:2006} studied the error rate of QMC and RQMC methods for such  unbounded integrands. In this paper, we generalize the results of \cite{owen:2006} to provide a rigorous error analysis for the IS estimator \eqref{eq:crudeQMC}. 
	
	\subsection{Two commonly used IS methods}
	
	How to choose $\bm\mu$ and $\bm \Sigma$ in the IS density? From the perspective of MC simulation, a good IS aims at reducing the variance $\sigma^2(\bm \mu,\bm \Sigma)$ given by \eqref{eq:MCvariance}
	as much as possible. Generally, it is hard to find the minimizer of the variance $\sigma^2(\bm \mu,\bm \Sigma)$. Assume that $G(\bm z)\ge 0$ for all $\bm z$. As a practical strategy, one may choose an IS density $p(\bm y;\bm \mu,\bm \Sigma)$ to mimic the behavior of the optimal (zero-variance) IS density
	$$
		p_{opt}(\bm z) := \frac{1}{C}G(\bm z)p(\bm z;\bm \mu_0,\bm \Sigma_0).
	$$
	Note that the optimal IS density involves the value of original integral, which is unknown. To this end, LapIS approximates the optimal IS density $p_{opt}(\bm z)$ by a quadrature Taylor approximation around its mode. Let $H(\bm z) =\log (p_{opt}(\bm z))$.
	Suppose that $H(\bm z)$ is differentiable and unimodal. Let $\bm \mu_{\star}$ be the mode of $H(\bm z)$, i.e.,
	$$
		\bm \mu_{\star} = \arg \max_{\bm z \in \mathbb{R}^d} H(\bm z)= \arg \max_{\bm z \in \mathbb{R}^d} G(\bm z)p(\bm z;\bm \mu_0,\bm \Sigma_0).
	$$
	Taking a second-order Taylor approximation around the mode $\bm \mu_{\star}$ gives
	$$
		H(\bm y) \approx H(\bm \mu_{\star}) - \frac{1}{2}(\bm y - \bm \mu_{\star})^T \bm \Sigma_{\star}^{-1}(\bm y - \bm \mu_{\star}),
	$$
	where
	\begin{equation}
		\label{eq:solveoptvar}
		\bm \Sigma_{\star} = -\nabla^2 H(\bm \mu_{\star})^{-1} = (\bm\Sigma_0 - \nabla^2\log(G(\bm \mu_{\star})))^{-1}.
	\end{equation}
	We then have
	$$
		p_{opt}(\bm z) = \exp (H(\bm z)) \approx \exp\left(H(\bm \mu_{\star})-\frac{1}{2}(\bm z - \bm \mu_{\star})^T \bm \Sigma_{\star}^{-1}(\bm z - \bm \mu _{\star})\right)\propto p(\bm z;\bm \mu_{\star},\bm \Sigma_{\star}). 
	$$	
	LapIS thus chooses $\bm \mu= \bm \mu_{\star}$ and $\bm \Sigma=\bm \Sigma_{\star}$ because $p(\bm z;\bm \mu_{\star},\bm \Sigma_{\star})$ is close to the optimal IS density (at least partially). 	
	On the other hand, ODIS chooses the drift $\bm \mu= \bm \mu_{\star}$ and leaves the covariance matrix unchanged, i.e.,  $\bm \Sigma=\bm \Sigma_0$. 
	
	How to assess the performance of IS? We should note that LapIS is not necessarily better than ODIS in the MC setting, depending on  how close the optimal IS density $p_{opt}(\bm z)$ is to a Gaussian density. LapIS is very effective if $p_{opt}(\bm z)$ is fitted well by a Gaussian distribution. On the other hand, in the QMC setting, there are many factors that may affect the performance of IS. In the following, we try to resolve this question for a class of integrands.
	
	\section{Main results}  \label{sec:main}
	This section provides rigorous theoretical results for RQMC-based IS estimates. To this end, we first provide RMSE of scrambled net quadrature for a class of unbounded integrands under a boundary growth condition introduced by \cite{owen:2006}. Let $(D_v G)(\bm z)$ denote
	the derivative of the function $G$ with respect to each $z_j$ once for all $j \in  v\subset 1{:}d := \{1,2,...,d\}$. If $v=\emptyset$, we make a convention that $(D_v G)(\bm z)=G(\bm z)$.  Let $1\{\cdot\}$ be an indicator function taking values in $\{0,1\}$.

\begin{theorem}\label{thm:mse}
	Let $f(\bm u)$ be a real-valued function defined over $(0,1)^d$ satisfying
	\begin{equation}\label{eq:grow}
		\abs{(D_vf)(\bm u)}\leq B\prod_{i=1}^d[\min(u_i,1-u_i)]^{-B_i-1\{i\in v\}}
	\end{equation}
	for some $B_i\in(0,1)$, some $B<\infty$ and all $v\subseteq 1{:}d$.
	Suppose that $\bm u_1,\dots,\bm u_N$ are a scrambled $(t,m,d)$-net in base $b\ge 2$ with $N=b^m$. Then the scrambled net quadrature yields an RMSE 
	$$\sqrt{\E{\left(\frac{1}{N}\sum_{i=1}^N f(\bm u_i)-\int_{(0,1)^d}f(\bm u) d \bm u\right)^2}}=O(N^{-1+\max_i B_i+\epsilon})$$ for arbitrarily small $\epsilon>0$.		
\end{theorem}

It is clear that mean error (i.e., $L^1$ error) is bounded above by RMSE (i.e., $L^2$ error). Consequently, Theorem~\ref{thm:mse} generalizes the result of \cite[Theorem 5.7]{owen:2006} which studied the mean error other than the RMSE. It may be of independent interest for RQMC integration with unbounded functions.
The proof of Theorem~\ref{thm:mse} is non-trivial, which is deferred to the Appendix. We should note that  \cite[Theorem 3.4]{gobet:hal-03631879} also provided the same RMSE rate, but under an assumption that the joint density between random pairs of sequence points is bounded above. The work \cite{gobet:hal-03631879} claimed that the assumption is satisfied for a $(0,m,d)$-net in base $b\ge 2$ by leveraging \cite[Theorem 3.6]{wiart:2021}. Our Theorem~\ref{thm:mse} does not require such an assumption and holds for general scrambled $(t,m,d)$-nets. The result in Theorem~\ref{thm:mse} can be easily extended to the first $N$ points of a scrambled $(t,d)$-sequence without requiring the constraint $N=b^m$ on the sample size.

The condition \eqref{eq:grow} is actually the second growth condition described in Owen \cite{owen:2006}. Owen \cite{owen:2006b} and Basu and Owen \cite{basu:2016b} studied other types of growth conditions for point singularities and singularities along a diagonal in the square, respectively.  When $\max_i B_i< 1/2$, then $f$ is square integrable and RQMC has a faster RMSE rate than MC. It is clear that large values of $B_i$ correspond to more severe singularities.
If all $B_i$ are arbitrarily small,  a nearly $O(N^{-1})$ error rate can be achieved. We may say that the condition \eqref{eq:grow} with arbitrarily small $B_i>0$ is friendly to QMC. 

In the following, we assume that the integrand $G(\bm z)$ is smooth enough and may be unbounded. We work on the case of $G(\bm z)$ having  QMC-friendly singularities as formally stated in Assumption~\ref{assum:bgc}.

\begin{assum}[Boundary growth condition]\label{assum:bgc}
	Suppose $G(\bm z)$ is a real-valued function such that for arbitrarily small $B_i>0$, there exists a constant $B>0$ such that
	\begin{equation}
		\abs{(D_v G)(\bm z)}\leq B\prod_{i=1}^d [1-\Phi(\abs{z_i})]^{-B_i},\label{eq:dgup}
	\end{equation}
	for any $v\subset 1{:}d$, then we say that $G$ satisfies the `QMC-friendly' boundary growth condition.
\end{assum}

We take scrambled $(t,m,d)$-net in base $b\ge 2$ with $N=b^m$ as RQMC by default. We thus do not specify the details of RQMC in the following statements. 
We continue to use the denotations of $I(\cdot)$, $\hat{I}_N(\cdot)$ defined in \eqref{eq:ISform} and \eqref{eq:crudeQMC}, respectively. Unless otherwise specified, $\hat{I}_N(\cdot)$ denotes the RQMC quadrature rule in the following.

\begin{theorem}\label{thm:convergence}
		If $G(\cdot)$ satisfies Assumption~\ref{assum:bgc}, then for arbitrarily small $\epsilon>0$,
		$$\sqrt{\mathbb{E}[(\hat{I}_N(G)-I(G))]}=O(N^{-1+\epsilon}).$$ 
\end{theorem}
	
\begin{proof}
Let $\bm z=\Phi^{-1}(\bm u)$ and $f(\bm u)=G(\Phi^{-1}(\bm u))$. For any $v\subset 1{:}d$, by Faa di Bruno formula \cite{cons:1996} we have
		$$(D_v f)(\bm u)= (D_v G)(\bm z)\prod_{i\in v}\frac{d\Phi^{-1}(u_i)}{d u_i}.$$
		By inverse function theorem, we have
		$$\frac{\partial \Phi^{-1}(u_i)}{\partial  u_i}= \frac{1}{\rho(\Phi^{-1}(u_i))}= \sqrt{2\pi}\exp((\Phi^{-1}(u_i))^2/2),
		$$
		where $$\rho(t)=\Phi^{'} (t)=\frac{1}{\sqrt{2\pi}}\exp(-t^2/2).$$
		Note the fact that $\Phi(t) \leq \exp(-t^2/2)$ for all $t \leq 0 $, then for all $0 < u_i \leq 1/2$,
		$$\exp[(\Phi^{-1}(u_i))^2/2] \leq \frac{1}{\Phi(\Phi^{-1}(u_i))}=\frac{1}{u_i}.$$ 
		Thus 
		$$\frac{\partial \Phi^{-1}(u_i)}{\partial  u_i} \leq \frac{\sqrt{2\pi}}{u_i}$$ for $u_i \in (0,1/2]$. Similarly for $u_i \in (1/2,1)$, we have 
		$$\exp[(\Phi^{-1}(u_i))^2/2] = \exp[(\Phi^{-1}(1-u_i))^2/2] \leq \frac{1}{1-u_i},$$ leading to 
		$$\frac{\partial \Phi^{-1}(u_i)}{\partial  u_i} \leq \frac{\sqrt{2\pi}}{1-u_i}$$ for $u_i \in (1/2,1)$. Therefore, we claim that 
		\begin{equation}
		\frac{\partial \Phi^{-1}(u_i)}{\partial u_i}=O([\min(u_i,1-u_i)]^{-1}).\label{eq:derphi}
		\end{equation}
		As a result, by \eqref{eq:dgup}, for arbitrarily small $B_i>0$, we have
		$$\abs{(D_v f)(\bm u)}=O\left(\prod_{i=1}^d[\min(u_i,1-u_i)]^{-B_i-1\{i\in v\}}\right),$$
		which verifies the condition \eqref{eq:grow}. 
		Applying Theorem~\ref{thm:mse} with arbitrarily small $B_i$ gives an RMSE of $O(N^{-1+\epsilon})$.
	\end{proof}
	
	\subsection{Non-negative integrands}

	Consider the integral \eqref{eq:target} with 
	$$	G(\bm z) \geq 0,\text{\ for all }\bm z.$$ 
	We perform IS as in \eqref{eq:ISform}. Let $\tilde G (\bm z)=G(\bm \mu+\bm L\bm z)$.  As a result, $G_{\is}(\bm z) = \tilde G (\bm z) W(\bm z)$. By the chain rule, it is easy to see that  
	\begin{equation}\label{eq:times}
		(D_u G_{\is})(\bm z)= \sum_{v\subset u} (D_v \tilde G)(\bm z) (D_{u-v} W)(\bm z).
	\end{equation}
	
	We next provide a sufficient condition for verifying  Assumption~\ref{assum:bgc} for the LR function $W(\bm z)$ given by \eqref{eq:weightf}. 
	
	\begin{theorem}\label{thm:weight}
		The LR function $W(\bm z)$  given by \eqref{eq:weightf} satisfies Assumption~\ref{assum:bgc} if and only if all the eigenvalues of the covariance matrix $\bm L^T\bm \Sigma_0^{-1}\bm L$ are larger than or equal to 1.
	\end{theorem}
	
	\begin{proof}
		Let $\bm C=\bm I_d-\bm L^T\bm \Sigma_0^{-1}\bm L$ and  $\bm \alpha=-\bm L^T\bm \Sigma_0^{-1}(\bm \mu-\bm \mu_0)$. By \eqref{eq:weightf}, we have 
		$$W(\bm z)  = \frac{\det (\bm \Sigma)^{1/2}}{\det (\bm \Sigma_0)^{1/2}}\exp\set{-\frac 12(\bm \mu-\bm\mu_0)^T\bm \Sigma_0^{-1}(\bm \mu-\bm\mu_0)}\exp\set{\frac 12 \bm z^T\bm C\bm z+\bm \alpha^T\bm z}.$$
		For any $u\subset 1{:}d$, $D_u W$ is a linear combination of terms of the form
		$$\prod_{i=1}^d z_i^{t_i}\exp\set{\frac 12 \bm z^T\bm C\bm z+\bm \alpha^T\bm z},$$
		where $t_i$ are nonnegative integers. Let $\bm C = \bm Q\bm \Lambda \bm Q^T$ be the singular value decomposition of the matrix $\bm C$, where $\bm Q$ is an orthogonal matrix, and $\bm \Lambda=\diag{\lambda_1,\dots,\lambda_d}$ is a diagonal matrix with $\lambda_1\ge \lambda_2\ge\dots \ge \lambda_d$. Now let $\tilde {\bm z}= \bm Q^T \bm z$, then we have $\exp\{\frac 12 \bm z^T\bm C\bm z\}=\exp\{\frac 12 \tilde {\bm z}^T\bm \Lambda \tilde {\bm z}\}.$ Note that $1-\lambda_i$ are the eigenvalues of  $\bm L^T\bm \Sigma_0^{-1}\bm L = \bm I_d-\bm C$.
		
		If all the eigenvalues of the matrix $\bm L^T\bm \Sigma_0^{-1}\bm L$ are larger than or equal to 1, we have $\lambda_i\le 0$ for all $i=1,\dots,d$, implying $\exp\{\frac 12 \bm z^T\bm C\bm z\}$ is bounded. Note that $$1-\Phi(z) = \Phi(-z) \leq \exp (-z^2/2),$$ for $z \geq 0$. Therefore, for arbitrarily small $B_i>0$, there exists $L>0$ such that
		$$
		\abs{\prod_{i=1}^d z_i^{t_i} \exp\{\bm \alpha^T\bm z\}}\leq L\prod_{i=1}^d \exp(B_iz_i^2/2) \leq L\prod_{i=1}^d (1-\Phi(\abs{z_i}))^{-B_i}.
		$$
		As a result, the function $W(\bm z) $ satisfies Assumption~\ref{assum:bgc}.
		
		Now suppose that the function $W(\bm z) $ satisfies Assumption~\ref{assum:bgc}. Letting $u=\emptyset$ in Assumption~\ref{assum:bgc}, we have 
		\begin{equation}
			\exp\left\lbrace\frac 12 \bm z^T\bm C\bm z+\bm \alpha^T\bm z\right\rbrace =\exp\left\lbrace\frac 12 \tilde {\bm z}^T\bm \Lambda \tilde {\bm z}+\bm \alpha^T\bm Q\tilde {\bm z}\right\rbrace = O\left(\prod_{i=1}^d  (1-\Phi(\abs{z_i}))^{-B_i}\right).\label{eq:terms}
		\end{equation}
		If there exists an eigenvalue of  $\bm L^T\bm \Sigma_0^{-1}\bm L$ which is smaller than 1, then we have $\lambda_1 > 0.$ Now let $ \tilde z_2=\dots=\tilde z_d=0$. Since $\bm z= \bm Q\tilde{\bm z}$ and the first column of $\bm Q$ is not a zero vector, there exist an index $k$ and constants $a_i$ such that $z_i = a_i z_k$. Then 
		$$\exp\left\lbrace\frac 12 \bm z^T\bm C\bm z+\bm \alpha^T\bm z\right\rbrace=\exp\left\lbrace\frac 12 \lambda_1 \tilde{z}_1^2+\bm \alpha^T\bm z\right\rbrace=\exp\{d_1 z_k^2+d_2 z_k\}$$
		for constants $d_1>0,d_2$. Let $d_3=\max_{i} (|a_i|)$ and $B=\sum_{i=1}^d B_i$. Then, taking $z_i = a_i z_k$, we find that
		\begin{align}
			\exp\left\lbrace\frac 12 \bm z^T\bm C\bm z+\bm \alpha^T\bm z\right\rbrace \prod_{i=1}^d (1-\Phi(\abs{z_i}))^{B_i}&=\exp\{d_1 z_k^2+d_2 z_k\}\prod_{i=1}^d (1-\Phi(\abs{a_iz_k}))^{B_i}\notag\\&\ge \exp\{d_1 z_k^2+d_2 z_k\}(1-\Phi(d_3\abs{z_k}))^{B}\notag
			\\&\ge \frac{\exp\{(d_1-Bd_3^2/2)z_k^2+d_2 z_k\}}{(2\pi)^{B/2} (d_3\abs{z_k} + 1/(d_3\abs{z_k}))^B},\label{eq:last}
		\end{align}
		where the last inequality follows from the inequality $1-\Phi(x) > \Phi' (x)/(x + 1/x)$ for $x > 0$ (see \cite{gord:1941}). If
		$B<2d_1/d_3^2$, then the right hand side of \eqref{eq:last}  goes to infinity as $\abs{z_k}\to \infty$. So \eqref{eq:terms} does not hold for the cases $z_i = a_i z_k$ with arbitrarily small $B_i>0$, leading to a contradiction. It follows that all eigenvalues of $\bm L^T\bm \Sigma_0^{-1}\bm L$ are not smaller than 1.
	\end{proof}
	
	\begin{theorem}\label{thm:convergencerate}
		If all the eigenvalues of the matrix $\bm L^T\bm \Sigma_0^{-1}\bm L$ are not smaller than 1 and $\tilde G (\bm z)=G(\bm \mu+\bm L\bm z)$ satisfies Assumption~\ref{assum:bgc}, then  for arbitrarily small $\epsilon>0$, $$\sqrt{\mathbb{E}[(\hat{I}_N(G_{\is})-I(G_{\is}))^2]}=O(N^{-1+\epsilon}).$$
	\end{theorem}

	\begin{proof}
		By Theorem~\ref{thm:weight},  $W(\bm z) $ satisfies Assumption~\ref{assum:bgc}.  It then follows from \eqref{eq:times} that $G_{\is}(\bm z)$ satisfies Assumption~\ref{assum:bgc}. Applying Theorem~\ref{thm:convergence} immediately completes the proof.
	\end{proof}
	\begin{cor}\label{cor:convergencerate}
		Suppose that $\bm \Sigma_0=\bm I_d$. If all the eigenvalues of the matrix $\bm \Sigma$ are not smaller than 1 and $\tilde G (\bm z)=G(\bm \mu+\bm L\bm z)$ satisfies Assumption~\ref{assum:bgc}, then  for arbitrarily small $\epsilon>0$, $$\sqrt{\mathbb{E}[(\hat{I}_N(G_{\is})-I(G_{\is}))^2]}=O(N^{-1+\epsilon}).$$
	\end{cor}
	\begin{proof}		
		If $\bm \Sigma_0=\bm I_d$, then $\bm L^T\bm \Sigma_0^{-1}\bm L=\bm L^T\bm L$ has the same eigenvalues of the  covariance matrix $\bm \Sigma=\bm L\bm L^T$. Applying Theorem~\ref{thm:convergencerate} completes the proof.
	\end{proof}
	
	\begin{rem}
		Theorem~\ref{thm:weight} shows that if the matrix $\bm L^T\bm \Sigma_0^{-1}\bm L$ has an eigenvalue smaller than 1, the LR function $W(\bm z) $ does not satisfy the `QMC-friendly' boundary growth condition. In other words, $W(\bm z)$ grows extraordinary fast when $\bm z$ goes to infinity. For this case, we may not expect a good performance of RQMC for the IS estimator $\hat{I}_N(G_{\is})$. This insight indicates that when we use IS associated with RQMC, picking a proper IS density is crucial. If we take $\bm \Sigma =\bm \Sigma_0$ (as in ODIS), then $\bm L^T \bm \Sigma_0^{-1} \bm L=\bm I_d$, the eigenvalues of $\bm L^T\bm \Sigma_0^{-1}\bm L$ are all ones. So if $\tilde{G}(\bm z)$ satisfies the `QMC-friendly' boundary growth condition, one can have a nearly $O(N^{-1})$ RMSE rate for RQMC integration. From this point of view, ODIS tends to be a safer choice.
	\end{rem}
	
	\begin{lem}\label{orthgonal}
		Let $\bm L_1$ and $\bm L_2$ be two square roots of the covariance matrix $\bm \Sigma$ satisfying $\bm L_1 \bm L_1^T = \bm L_2 \bm L_2^T = \bm \Sigma$. Then $\bm L_1^T \bm \Sigma_0^{-1} \bm L_1$ has the same eigenvalues of $\bm L_2^T \bm \Sigma_0^{-1} \bm L_2$. 
	\end{lem}
	\begin{proof}
		Since both $\bm L_1$ and $\bm L_2$ are square roots of $\bm \Sigma$, $\bm U := \bm L_1^{-1} \bm L_2$ is an orthogonal matrix. Then we have $\bm L_2^T \bm \Sigma_0^{-1} \bm L_2 = \bm U^T \bm L_1^T \bm \Sigma_0^{-1} \bm L_1 \bm U$, which shares the same eigenvalues of $\bm L_1^T \bm \Sigma_0^{-1} \bm L_1$.
	\end{proof}
	
	Lemma~\ref{orthgonal} shows that although the square root of $\bm \Sigma$ is not unique, the eigenvalues of the matrix $\bm L^T\bm \Sigma_0^{-1}\bm L$ do not depend on the choice of $\bm L$. On the other hand, it is clear that the equality \eqref{eq:weightf} also holds  if $G_{\is}(\bm z)$ is replaced by $G_{\is}(\bm U\bm z)$ for an arbitrary orthogonal matrix $\bm U$. It has the same effect as replacing $\bm L$ by $\bm L\bm U$. In the MC setting the variance of $G_{\is}(\bm U\bm z)$ is irrelevant to the choice of the orthogonal matrix $\bm U$. However, the choice of $\bm U$ is crucial for QMC quadrature rules because the orthogonal matrix $\bm U$ has an impact on the effective dimension of the integrand $G_{\is}(\bm U\bm z)$ which is usually served as an indicator of the performance of QMC \cite{caflisch1997valuation}. It is possible to choose a suitable $\bm U$ such that the resulting integrand $G_{\is}(\bm U\bm z)$ has a lower effective dimension even when the nominal dimension $d$ is large. To overcome the impact of high dimensionality, some dimension reduction strategies are proposed to find a good $\bm U$ in the literature \cite{wang2013pricing,xiao:wang:2017}. Interestingly, our finding ensures that applying dimension reduction strategies does not change the RMSE rate $O(N^{-1+\epsilon})$ established in Theorem~\ref{thm:convergencerate}.

	\subsection{Positivization}
	If $G(\bm z)$ has mixed sign, one may use positivization. Owen and Zhou \cite{owen:2000} proposed to use partition of identity. Define a partition of the identity by a set of functions, $v_j,j=1,\dots,r,$ satisfying 
	$$z = \sum_{j=1}^r v_j(z),z\in\mathbb{R}.$$
	Moreover, $v_j$ does not have mixed sign. A smooth partition of identity can be achieved by
	$$\frac z 2\pm \sqrt{\eta+z^2/4},$$
	where $\eta>0$. Let $v_+(z) = \frac z 2+ \sqrt{\eta+z^2/4},\ v_-(z) = -\frac z 2+ \sqrt{\eta+z^2/4}$.
	We thus have 
	$$C = \int v_+(G(\bm z)) p(\bm z;\bm \mu_0,\bm \Sigma_0)d\bm z- \int v_-(G(\bm z))p(\bm z;\bm \mu_0,\bm \Sigma_0)d\bm z.$$
	Note that $v_\pm(G(\bm z))\ge 0$. We use IS for each part. That is,
	\begin{align*}
		C = &\int v_+(G(\bm z)) \frac{p(\bm z;\bm \mu_0,\bm \Sigma_0)}{p(\bm z;\bm \mu_+,\bm \Sigma_+)}p(\bm z;\bm \mu_+,\bm \Sigma_+)d\bm z\\&- \int v_-(G(\bm z)) \frac{p(\bm z;\bm \mu_0,\bm \Sigma_0)}{p(\bm z;\bm \mu_-,\bm \Sigma_-)}p(\bm z;\bm \mu_-,\bm \Sigma_-)d\bm z.
	\end{align*}
	Let 
	\begin{equation*}
		G_{\is}^+(\bm z)=v_+(G(\bm \mu_++\bm L_+\bm z))W(\bm z;\bm \mu_+,\bm L_+)
	\end{equation*}
    and
    \begin{equation*}
		G_{\is}^-(\bm z)=v_-(G(\bm \mu_-+\bm L_-\bm z))W(\bm z;\bm \mu_-,\bm L_-),
	\end{equation*}
	where $W(\bm z;\bm L,\bm \Sigma)$ is given by \eqref{eq:weightf}, $\bm L_+\bm L_+^T=\bm\Sigma_+$ and $\bm L_-\bm L_-^T=\bm\Sigma_-$. As a result, $C=I(G_{\is}^+)-I(G_{\is}^-)$. The corresponding RQMC estimate is given by $\hat I_N(G_{\is}^+-G_{\is}^-)=\hat{I}_N(G_{\is}^+)-\hat{I}_N(G_{\is}^-)$, where the two integrations use common random inputs.
	
	\begin{theorem}
		If all the eigenvalues of the covariance matrix $\bm L^T\bm \Sigma_0^{-1}\bm L$ are not smaller than 1 and $\tilde G (\bm z)=G(\bm \mu+\bm L\bm z)$ satisfies Assumption~\ref{assum:bgc} for $\bm \mu=\bm \mu_{\pm}$ and $\bm L=\bm L_{\pm}$, then $$\sqrt{\mathbb{E}[(\hat I_N(G_{\is}^+-G_{\is}^-)-C)^2]}=O(N^{-1+\epsilon}).$$
	\end{theorem}
	\begin{proof}
		Note that $\mathbb{E}[(\hat I_N(G_{\is}^+-G_{\is}^-)-C)^2]\le 2\mathbb{E}[(\hat I_N(G_{\is}^+)-I(G_{\is}^+))^2]+2\mathbb{E}[(\hat I_N(G_{\is}^-)-I(G_{\is}^-))^2]$. It suffices to prove that both $\hat I_N(G_{\is}^+)$ and $\hat I_N(G_{\is}^-)$ have an RMSE of $O(N^{-1+\epsilon})$. Let $H(\bm z) = v_+(\tilde G (\bm z))$. For any $u\subset 1{:}d$, we have $D_u H(\bm z)$  is a linear combination of terms of the form $$v_+^{(k)}(\tilde G (\bm z))\prod_{i\in I}D_{w_i}\tilde G (\bm z),$$
		where $k\ge 0$, $w_i\subset u$ and $w_i\cap w_j=\emptyset$ for any $i\neq j$. Note that $v_+^{(k)}(z)$ is bounded for any $k\ge 1$. Since $v_+(z)\le \sqrt{4\eta + z^2}$, $$v_+^{(0)}(\tilde G (\bm z))=v_+(\tilde G (\bm z))\le \sqrt{4\eta + \tilde G (\bm z)^2}.$$ Since $\tilde G (\bm z)$ satisfies  Assumption~\ref{assum:bgc}, $H(\bm z)$ also satisfies  Assumption~\ref{assum:bgc}. By Theorem~\ref{thm:weight},  $W(\bm z;\bm \mu_+,\bm L_+) $ satisfies Assumption~\ref{assum:bgc}. It then follows from \eqref{eq:times} that $G_{\is}^+$ satisfies Assumption~\ref{assum:bgc}. Applying Theorem~\ref{thm:convergence} immediately leads to $$\sqrt{\mathbb{E}[(\hat I_N(G_{\is}^+)-I(G_{\is}^+))^2]}=O(N^{-1+\epsilon}).$$ Similarly, $\hat I_N(G_{\is}^-)$ has an RMSE of $O(N^{-1+\epsilon})$.
	\end{proof}
	
	A good property of the smooth positivization $v_{\pm}(z)$ is that it does not destroy the boundary growth condition. So the RMSE rate $O(N^{-1+\epsilon})$ can be retained after the positivization. This property does not hold if we take the positive part $\max(G(\bm z),0)$ and the negative part $\max(-G(\bm z),0)$ of the function $G(\bm z)$ because the existence of kinks.
	
	\section{Multivariate $t$ distribution as the proposal}\label{sec:mult}
	
	In this section, we take a multivariate $t$ distribution as the proposal of IS for the problem \eqref{eq:target}. The multivariate $t$ distribution with center $\bm\mu\in\mathbb{R}^{d\times1}$, scale (positive definite) matrix $\bm\Sigma\in\mathbb{R}^{d\times d}$ and $\nu>0$ degrees of freedom, denoted by $t_\nu(\bm\mu,\bm\Sigma)$ has a representation
	$$\bm t = \bm \mu + \bm L\bm x,$$
	where  $\bm L\bm L^T = \bm \Sigma$, and
	\begin{equation}\label{eq:standardt}
		\ \bm x:=\psi(\bm z)=\frac{\bm z_{1{:}d}}{\sqrt{z_{d+1}/\nu}},
	\end{equation}
  $\bm z_{1{:}d}\sim N(\bm 0,\bm I_d)$ independently of $z_{d+1}\sim \chi_{\nu}^2$. We should note that $\bm x$ defined by \eqref{eq:standardt} is the standard multivariate $t$. However,  the components of $\bm x$ are not independent. The multivariate $t$ distribution has a density given by
	\begin{equation*}
		q(\bm t;\bm\mu,\bm\Sigma,\nu) = c_{\bm\mu,\bm\Sigma,\nu}(1+(\bm t-\bm \mu)^T\bm\Sigma^{-1}(\bm t-\bm u)/\nu)^{-(\nu+d)/2},
	\end{equation*}
	where 
	$$c_{\bm\mu,\bm\Sigma,\nu} = \frac{\Gamma((\nu+d)/2)}{\abs{\bm \Sigma}^{1/2}(\nu\pi)^{d/2}\Gamma(\nu/2)}.$$
	
	By a change of measure and a change of variables, the integral \eqref{eq:target} can be changed to
	\begin{eqnarray}
		C &:=& \int_{\mathbb{R}^d} G(\bm t) \frac{p(\bm t;\bm \mu_0,\bm \Sigma_0)}{q(\bm t;\bm\mu,\bm\Sigma,\nu)}q(\bm t;\bm\mu,\bm\Sigma,\nu)d\bm t \nonumber\\ 
		&=& \int_{\mathbb{R}^d} G(\bm \mu+\bm L\bm x) \frac{p(\bm \mu+\bm L\bm x;\bm \mu_0,\bm \Sigma_0)}{q(\bm \mu+\bm L\bm x;\bm \mu,\bm \Sigma,\nu)}q(\bm x;\bm 0,\bm I_d,\nu)d\bm x\notag\\
		&=& \int_{\mathbb{R}^d} G(\bm \mu+\bm L\bm x) W(\bm x)q(\bm x;\bm 0,\bm I_d,\nu)d\bm x\notag\\
		&=& \int_{\mathbb{R}^d} \tilde G(\bm x) W(\bm x)q(\bm x;\bm 0,\bm I_d,\nu)d\bm x\notag,
	\end{eqnarray}
	where $\tilde G(\bm x)=G(\bm \mu+\bm L\bm x)$, the LR function
	\begin{align}
	W(\bm x) &= \frac{p(\bm\mu+\bm L\bm x;\bm \mu_0,\bm\Sigma_0)}{q(\bm\mu+\bm L\bm x;\bm\mu,\bm\Sigma,\nu)}\notag\\ &\propto (1+\bm x^T\bm x/\nu)^{(\nu+d)/2}\exp\{-(\bm\mu+\bm L\bm x-\bm\mu_0)^T\bm\Sigma_0^{-1}(\bm\mu+\bm L\bm x-\bm\mu_0)/2\}\notag\\
	&\propto (1+\bm x^T\bm x/\nu)^{(\nu+d)/2}\exp\set{-\frac 12 \bm x^T(\bm L^T\bm \Sigma_0^{-1}\bm L)\bm x+\bm \alpha^T\bm x}\label{weit},
\end{align}
and $\bm \alpha=-\bm L^T\bm \Sigma_0^{-1}(\bm \mu-\bm \mu_0)$. We should note that any mixed partial derivative of $W(\bm x)$ is bounded since $\bm L^T\bm \Sigma_0^{-1}\bm L$ is positive definite.

Now the integral \eqref{eq:target} is transformed into an expectation of $\tilde G(\bm x) W(\bm x)$ with respect to $\bm x\sim t_\nu(\bm 0,\bm I_d)$. By \eqref{eq:standardt}, the integral \eqref{eq:target} can be further transformed into an expectation of 
\begin{equation*}\label{eq:newGis}
G_\is(\bm z):=\tilde G(\psi(\bm z)) W(\psi(\bm z)).
\end{equation*}
Denote $\mathrm{Gam}(\alpha,\theta)$ as the gamma distribution with a shape parameter $\alpha>0$ and a scale parameter $\theta>0$. Obviously, $z_{d+1}\sim \mathrm{Gam}(\nu/2,2)$.
Define the lower incomplete gamma function as
\begin{equation*}
	\gamma_{\alpha}(x) = \int_0^x t^{\alpha-1}e^{-t}d t.\label{eq:ingamma}
\end{equation*}
The $\mathrm{Gam}(\alpha,1)$ has a CDF $\gamma_{\alpha}(x)/\Gamma(\alpha)$ for all $x\ge0$, where $\Gamma(\alpha)=\gamma_{\alpha}(+\infty)$ is the Gamma function.
To simulate $\bm z$ via uniform random variables, we take
\begin{equation}\label{eq:tau}
	\bm z := \tau(\bm u) = (\tau_1(u_1),\dots,\tau_{d+1}(u_{d+1}))^T =(\Phi^{-1}(u_{1{:}d}),2\gamma_{\nu/2}^{-1}(\Gamma(\nu/2)u_{d+1}))^T,
\end{equation}
where $\bm u=(u_1,\dots,u_{d+1})\sim U(0,1)^{d+1}$ and we use the fact that $2\gamma_{\nu/2}^{-1}(\Gamma(\nu/2)u_{d+1})\sim \mathrm{Gam}(\nu/2,2)$.
This yields an RQMC estimator for the integral \eqref{eq:target}
\begin{equation}\label{eq:testmator}
	\hat{I}_N(G_\is) = \frac 1 N \sum_{i=1}^N G_\is(\tau(\bm u_i)),
\end{equation}
where $\bm u_1,\dots,\bm u_N$ are $N$ RQMC points in $(0,1)^{d+1}$.

Let $\mathbb{N}_0$ be the set of nonnegative integers. For $\bm \lambda=(\lambda_1,\dots,\lambda_d)\in \mathbb{N}_0^d$, define $(\partial ^{\bm\lambda} g)(\bm x)$ as the mixed partial derivative of $g(\bm x)$ taken $\lambda_i$ times with respect to $x_i$, and define $\abs{\bm \lambda}=\sum_{i=1}^d\lambda_i$.

\begin{theorem}\label{thm:trmse}
Let $\tilde{G}(\bm x)=G(\bm \mu+\bm L\bm x)$. If 
\begin{equation}\label{eq:newmixderi}
\abs{(\partial ^{\bm\lambda} \tilde G)(\bm x)\big|_{\bm x=\psi(\tau(\bm u))}}=O\left(\prod_{i=1}^{d+1}[\min(u_i,1-u_i)]^{-A_i}\right)
\end{equation}
for some $A_i>0$ and any $\bm\lambda\in\mathbb{N}_0^d$ with $|\bm\lambda|\le d+1$, then the RQMC estimator $\hat{I}_N(G_\is)$ given by \eqref{eq:testmator} has an RMSE of $O(N^{-1+\max_i \tilde{A_i}+\epsilon})$ for arbitrarily small $\epsilon>0$, where  $\tilde{A}_i=A_i,i=1,\dots,d$ and $\tilde{A}_{d+1}=A_{d+1}+(d+3)/\nu$.
\end{theorem}

\begin{proof}
Let $g(\bm x)=\tilde{G}(\bm x)W(\bm x)$. So $G_\is(\bm z)=g (\psi(\bm z))$, where $\bm x=\psi(\bm z)$ with $$\psi(\bm z)=(\psi_1(\bm z),\dots,\psi_d(\bm z))^\top:\mathbb{R}^{d+1}\mapsto \mathbb{R}^{d}$$ given by \eqref{eq:standardt}. Using the Faa di Bruno formula for mixed partial derivatives taken at most once with respect to every index (see Equation (10) of \cite{basu:2016}), for any $v\subset 1{:}(d+1)$,  one gets 
\begin{equation}\label{eq:DvGIS}
(D_v G_\is)(\bm z)=\sum_{\bm\lambda \in \mathbb{N}_0^d,1\le|\bm\lambda|\le|v|}(\partial^{\bm\lambda} g)(\bm x)\sum_{s=1}^{\abs{v}}\sum_{(\ell_r,k_r)\in \mathrm{KL}(s,v,\bm\lambda)}\prod_{r=1}^sD_{\ell_r}\psi_{k_r}(\bm z),
\end{equation}
where 
\begin{align*}
	\mathrm{KL}(s,v,\bm\lambda) &= \{(\ell_r,k_r), r\in 1{:}s| \ell_r  \subset 1{:}(d+1),\ k_r\in 1{:}d,\ \cup_{r=1}^s\ell_r=v,\\&\ell_r\neq \emptyset,\ell_r\cap\ell_{r'}=\emptyset \text{ for } r\neq r'\text{ and } |\{j\in 1{:}s|k_j=i\}|=\lambda_i \}.
\end{align*}

Since any mixed partial derivative of $W(\bm x)$ is bounded, by \eqref{eq:newmixderi}, we have
$$|(\partial^{\bm\lambda} g)(\bm x)|=|(\partial^{\bm\lambda} g)(\psi(\tau(\bm u)))|=O\left(\prod_{i=1}^{d+1}[\min(u_i,1-u_i)]^{-A_i}\right).$$
We next bound $\prod_{r=1}^sD_{\ell_r}\psi_{k_r}(\bm z)$. Note that for $j=1,\dots,d$,
$$\psi_{j}(\bm z)=\frac{z_j}{\sqrt{z_{d+1}/
\nu}}.$$
For any nonempty $u\subset 1{:}(d+1)$, we have
\begin{equation*}
D_u\psi_{j}(\bm z)= 
\begin{cases}
\sqrt{\nu}z_{d+1}^{-1/2},	& u = \{j\},\\
(-1/2)\sqrt{\nu}z_jz_{d+1}^{-3/2},	& u = \{d+1\},\\
(-1/2)\sqrt{\nu}z_{d+1}^{-3/2},	& u = \{j,d+1\},\\
0,	& \text{else}.
	\end{cases}
\end{equation*}

By \eqref{eq:tau}, we have
$$
z_{d+1} = 2\gamma_{\nu/2}^{-1}(\Gamma(\nu/2)u_{d+1}).
$$
By L'Hospital's rule, we find that $$\lim_{x\to 0}\frac{\gamma_{\alpha}(x)}{x^\alpha}=\lim_{x\to 0}\frac{\gamma'_{\alpha}(x)}{\alpha x^{\alpha-1}}=1/\alpha.$$ 
Then
$$\lim_{u\to 0}\frac{\gamma_{\alpha}(\gamma_{\alpha}^{-1}(\Gamma(\alpha)u))}{\gamma_{\alpha}^{-1}(\Gamma(\alpha)u)^\alpha}=\lim_{u\to 0}\frac{\Gamma(\alpha)u}{\gamma_{\alpha}^{-1}(\Gamma(\alpha)u)^\alpha}=\frac 1\alpha,$$
implying $1/\gamma_{\alpha}^{-1}(\Gamma(\alpha)u)=O(u^{-1/\alpha})$.
Then for any $\beta\ge 0$,
\begin{equation}
z_{d+1}^{-\beta}=2^{-\beta}(\gamma_{\nu/2}^{-1}(\Gamma(\nu/2)u_{d+1}))^{-\beta} =O(u_{d+1}^{-2\beta/\nu})=O([\min(u_{d+1},1-u_{d+1})]^{-2\beta/\nu}).\label{eq:zdbeta}
\end{equation}
Also, we have $|z_j|=|\Phi^{-1}(u_j)|=O([\min(u_j,1-u_j)]^{-B_j})$ for arbitrarily small $B_j>0$ and $j=1,\dots,d$.
As a result, 
\begin{align*}
\abs{\prod_{r=1}^sD_{\ell_r}\psi_{k_r}(\bm z)}&\le \prod_{r=1}^s \left(\sqrt{v} z_{d+1}^{-1/2-1\{d+1\in \ell_r\}} |z_{k_r}|^{1\{k_r\notin \ell_r\}}\right)\\
&=O\left(\prod_{j=1}^d[\min(u_j,1-u_j)]^{-B_j}[\min(u_{d+1},1-u_{d+1})]^{-(d+3)/\nu}\right).
\end{align*}
Using \eqref{eq:DvGIS} gives
\begin{equation}\label{eq:dvG}
\abs{D_v G_\is(\bm z)}=O\left(\prod_{j=1}^d[\min(u_j,1-u_j)]^{-B_j}[\min(u_{d+1},1-u_{d+1})]^{-(d+3)/\nu}\right).
\end{equation}

Let $f(\bm u)=G_\is(\tau(\bm u))$. Then for any $v\subset 1{:}(d+1)$, we have
\begin{equation}\label{eq:zd1}
	(D_v f)(\bm u)=(D_v G_\is)(\bm z)\prod_{i\in v}\frac{d\tau_i(u_i)}{du_i}.
\end{equation}
By \eqref{eq:derphi}, we have $$\abs{\frac{d\tau_{j}(u_{j})}{du_{j}}}=\abs{\frac{d\Phi^{-1}(u_{j})}{du_{j}}}=O([\min(u_j,1-u_j)]^{-1}),\ j=1,\dots,d.$$
On the other hand, 
\begin{align*}
	\frac{d\tau_{d+1}(u_{d+1})}{du_{d+1}}&=\frac{2\Gamma(\nu/2)}{\gamma_{\nu/2}'(\gamma_{\nu/2}^{-1}(\Gamma(\nu/2)u_{d+1}))}\\&=\frac{2\Gamma(\nu/2)}{\gamma_{\nu/2}'(z_{d+1}/2)}\\&=2\Gamma(\nu/2)(z_{d+1}/2)^{1-\nu/2}e^{z_{d+1}/2}.
\end{align*}
Recall that $z_{d+1}=\tau_{d+1}(u_{d+1})$. By \eqref{eq:zdbeta}, as $u_{d+1}\to0$, $$\abs{\frac{d\tau_{d+1}(u_{d+1})}{du_{d+1}}}=O(u_{d+1}^{-1+2/\nu})=O(u_{d+1}^{-1}).$$
By L'Hospital's rule, it is easy to see that  $$\lim_{x\to\infty}\frac{\Gamma(\alpha)-\gamma_{\alpha}(x)}{\gamma'_{\alpha}(x)}=\lim_{x\to\infty}\frac{-\gamma_{\alpha}'(x)}{\gamma^{''}_{\alpha}(x)}=1.$$ It then follows
$$\lim_{u\to1}\frac{\Gamma(\alpha)-\gamma_{\alpha}(\gamma^{-1}_{\alpha}(\Gamma(\alpha)u))}{\gamma'_{\alpha}(\gamma^{-1}_{\alpha}(\Gamma(\alpha)u))}=\lim_{u\to1}\frac{\Gamma(\alpha)(1-u)}{\gamma'_{\alpha}(\gamma^{-1}_{\alpha}(\Gamma(\alpha)u))}=1.$$
As $u_{d+1}\to 1$, we have
$$
\abs{\frac{d\tau_{d+1}(u_{d+1})}{du_{d+1}}}=\frac{2\Gamma(\nu/2)}{\gamma_{\nu/2}'(\gamma_{\nu/2}^{-1}(\Gamma(\nu/2)u_{d+1}))}=O((1-u_{d+1})^{-1}).
$$
Therefore,
$$
\abs{\frac{d\tau_{d+1}(u_{d+1})}{du_{d+1}}}=O([\min(u_{d+1},1-u_{d+1})]^{-1}).
$$
By \eqref{eq:dvG} and \eqref{eq:zd1}, we have
$$|(D_v f)(\bm u)|=O\left(\prod_{i=1}^{d+1} [\min(u_i,1-u_i)]^{-\tilde B_i-1\{i\in v\}}\right),$$
where $\tilde B_j=A_j+B_j,j=1,\dots,d$ and $B_{d+1}= A_{d+1}+(d+3)/\nu$. Since $B_j$ are arbitrarily small, applying Theorem~\ref{thm:mse} with $\tilde B_j$ completes the proof.
\end{proof}

Differently from Theorem~\ref{thm:convergencerate}, Theorem~\ref{thm:trmse} holds for any positive definite matrix $\bm \Sigma_0$ and any multivariate $t$ distribution, including $t_{\nu}(\bm \mu_\star,\bm I_d)$ and $t_{\nu}(\bm \mu_\star,\bm \Sigma_\star)$ which result from the ODIS and LapIS, respectively. This is because any mixed partial derivative of the LR function \eqref{weit} is bounded.
However, the rate established in Theorem~\ref{thm:trmse} is conservative. It
favors large degree of freedom $\nu$, especially $\nu\gg d$. Indeed, as $\nu\to\infty$, the rate reduces to the rate established in Theorem~\ref{thm:convergencerate} if $A_i$ are arbitrarily small. In our numerical experiments, we look at the performance of RQMC for a small $\nu$. A tight RMSE upper bound for multivariate $t$ distributions is left for future research.

\section{Examples}\label{sec:exam}
	
\subsection{Rendleman-Bartter Model}
	
	Caflisch \cite{Caflisch1998} considered the problem of valuing a discount bond, where the interest rates were assumed to follow the Rendleman-Bartter model \cite{Hull2011}. Here we consider the fair price of a one-year zero coupon bond with a face value of \$$1$, which can be represented as a Gaussian integral
	$$C = \int_{\mathbb{R}} G(z)p(z;0,1)dz,$$	
	where 
	\begin{equation}
		\label{eq:defGinex1}
		G(z) = \frac{1}{1+r_1(z)}
	\end{equation}
	with $r_1(z)=r_0\exp(-\sigma^2/2+\sigma z)=:c\exp(\sigma z)$ and
	$c=r_0\exp(-\sigma^2/2)$, $r_0$ and $r_1$ denote interest rates at time
	$t=0,1$, respectively, and $\sigma$ denotes the volatility. Consider the dimension $d=1$ for this example. Let
	$$F(z) = \log G(z) = -\log(1+r_1(z)).$$
	
	In the following we focus on ODIS and LapIS. We first find $\mu_{\star}$ by solving
	$F^{'}(\mu_{\star}) = \mu_{\star}$ or equivalently
	\begin{equation}
		\label{eq:solution of mu}
		-\frac{\sigma r_1(\mu_{\star})}{1+r_1(\mu_{\star})}=\mu_{\star},
	\end{equation}
	and then find the  variance $\Sigma_{\star}$ for LapIS  given by
	\begin{equation}\label{eq:solution of sigma}
		\Sigma_{\star} = (1-F^{''}(\mu_{\star}))^{-1},
	\end{equation}
	where
	\begin{equation}\label{eq:form of f''}
		F^{''}(\mu_{\star}) = -\frac{\sigma^2 r_1(\mu_{\star})}{(1+r_1(\mu_{\star}))^2}.
	\end{equation}
	\begin{lem}
		\label{lem:ex1}
		The optimal drift $\mu_{\star}$ and the variance $\Sigma_{\star}$ have the following bounds:
		$$
			-\sigma < \mu_{\star} < 0,
		$$
		$$
			0<\Sigma_{\star} < 1.
		$$
	\end{lem}
	\begin{proof}
		The bounds of $\mu_{\star}$ follow from \eqref{eq:solution of mu}. By
		\eqref{eq:form of f''} we have $F^{''}(\mu_{\star})<0$. It follows from \eqref{eq:solution of sigma} that $\Sigma_{\star}\in(0,1)$.
	\end{proof}
	
	Letting $L = \sqrt{\Sigma_{\star}}$ and noting $\mu_0=0$ and $\Sigma_0 = 1$, we have $L\Sigma_0^{-1}L=\Sigma_{\star} < 1$. It follows from Theorem~\ref{thm:weight} that $W(z;\mu_{\star},\Sigma_{\star})$ does not satisfy Assumption~\ref{assum:bgc}. Thus LapIS fails to satisfy the condition of Theorem~\ref{thm:convergencerate}. On the other hand, both $G$ and
	$dG/dz$ are bounded since $$|G| \leq 1$$ and
	$$\left\lvert \frac{dG}{dz}\right\rvert=\left\lvert\frac{\sigma}{2+r_1+1/r_1}\right\rvert
	\leq \frac{\sigma}4.$$ It follows from Theorem~\ref{thm:convergencerate} that ODIS
	has an RMSE of $O(N^{-1+\epsilon})$.
	Although the LR function $W(z;\mu_{\star},\Sigma_{\star})$ does not satisfy the boundary growth condition, it is unclear whether the LapIS estimator $G_{\is}(z)$ violates the boundary growth condition. For this example, we are able to work out tail behaviors of $G_{\is}(z)$ for the two IS methods.
	\begin{lem} Suppose $G(z)$ is given by \eqref{eq:defGinex1}. Then for LapIS,
		$$
			G_{\is}(z)=G(z)W(z;\mu_{\star},\Sigma_{\star}) \apprge \exp(\kappa z^2), \quad  \text{as }z \rightarrow \pm\infty
		$$
		for any $\kappa\in (0,(1-\Sigma_{\star})/2)$, implying $\lim_{z\to \pm\infty}G_{\is}(z)=\infty$.
		For ODIS, we have
		$$
			G_{\is}(z)=G(z)W(z;\mu_{\star},1) \apprle \left\{
			\begin{aligned}
				\exp(-\mu_{\star}z), \quad  \text{as }z \rightarrow -\infty, \\
				\exp(-(\mu_{\star}+\sigma)z), \quad  \text{as }z \rightarrow +\infty,
			\end{aligned}
			\right.  	
		$$
		implying $\lim_{z\to \pm\infty}G_{\is}(z)=0$.
	\end{lem}
	\begin{proof}
		First consider the LapIS. By \eqref{eq:weightf} and \eqref{eq:defGinex1}, we have
		\begin{align*}
			G_{\is}(z)&=G(z)W(z;\mu_{\star},\Sigma_{\star}) \\&= \frac{\sqrt{\Sigma_{\star}}}{1+c\exp(\sigma z)} \exp\left(\frac{1-\Sigma_{\star}}{2}z^2-\mu_{\star}\sqrt{\Sigma_{\star}}z-\frac{\mu_{\star}^2}2\right)
			\apprge \exp(\kappa z^2)
		\end{align*}
		for any $\kappa\in (0,(1-\Sigma_{\star})/2)$ as $z \rightarrow \pm\infty$.
		
		Now consider the case of ODIS, in which
		\begin{align*}
			G_{\is}(z)&=G(z)W(z;\mu_{\star},1) = \frac{1}{1+c\exp(\sigma z)} \exp\left(-\mu_{\star}z-\frac{\mu_{\star}^2}2\right).
		\end{align*}
		As $z \rightarrow -\infty$, we have $G_{\is}(z)\apprle \exp(-\mu_{\star}z)\to 0$ since $\mu_{\star}<0$ by Lemma \ref{lem:ex1}.
		On the other hand, when $z \rightarrow +\infty$, $G_{\is}(z)\apprle \exp(-(\mu_{\star}+\sigma) z)\to 0$ since $\mu_{\star}+\sigma > 0$ by Lemma~\ref{lem:ex1}.            
	\end{proof}
	
	The lemma above indicates that LapIS may introduce ``severe" singularities along boundaries of the unit cube (when using $z=\Phi^{-1}(u)$), which grow faster than the RHS of \eqref{eq:dgup}. This would make QMC inefficient.

	\subsection{Bayesian Logistic Regression}
	
	In this example, we consider integrals with respect to the posterior distribution under the  Bayesian logistic regression model. Let $\bm{X}=(\bm x_1,\dots,\bm x_n)$ denote the predictors,  $\bm Y=(Y_1,\dots,Y_n)$ denote independent response observations, and $\bm z\in \mathbb{R}^d$ denote the vector of unknown parameters. The predictors are assumed to be bounded and not all zero. Every binary random variable $Y_i \in \{0,1\}$ is related to the predictor $\bm x_i \in \mathbb{R}^d$ by 
	$$
	\mathbb{P}(Y_i=1)=\frac{\exp(\bm x_i^T\bm z)}{1+\exp(\bm x_i^T\bm z)}.
	$$
	The likelihood function is then given by
	$$
		\ell(\bm z) = \prod_{i=1}^n\left(\frac{\exp(\bm x_i^T\bm z)}{1+\exp(\bm x_i^T\bm z)}\right)^{Y_i}\left(\frac{1}{1+\exp(\bm x_i^T\bm z)}\right)^{1-Y_i}.
	$$
	Under the Bayesian framework, we model the parameter $\bm z$ as a random vector, whose prior distribution is assumed to be a standard normal distribution $N(\bm 0,\bm I_d)$. Then the posterior distribution of $\bm z$ given $\bm Y$ is
	$$
		\pi(\bm z) = \frac{1}{c}p(\bm z;\bm 0,\bm I_d)\ell(\bm z),	
	$$
	where the normalizing constant $c$ is rarely known.
	
	Our goal is to compute the posterior expectation $\mathbb{E}_{\pi(\bm z)}[f(\bm z)]$, where $f$ is the test function. For example, if $f(\bm z)=z_i$, the expectation is known as the posterior mean. Since the normalizing constant $c$ is unknown, we may write the posterior expectation as a ratio of two integrals with respect to the prior distribution $p(\bm z;\bm 0,\bm I_d)$ given by
	\begin{equation}\label{eq:posterior expectation}
		\mathbb{E}_{\pi(\bm z)}[f(\bm z)]=\frac{\int_{\mathbb{R}^d}f(\bm z)\ell(\bm z)p(\bm z;\bm 0,\bm I_d)d\bm z}{\int_{\mathbb{R}^d}\ell(\bm z)p(\bm z;\bm 0,\bm I_d)d \bm z}.
	\end{equation}
	We then apply  IS for the two integrals in \eqref{eq:posterior expectation}, such as ODIS and LapIS. The final ratio estimator is given by
	\begin{equation}
		\hat R_N = \frac{\hat{I}_N(G_{\is}^{\mathrm{num}})}{\hat{I}_N(G_{\is}^{\mathrm{den}})}=\frac{\frac 1 N \sum_{i=1}^N  G_{\is}^{\mathrm{num}}(\Phi^{-1}(\bm u_i))}{\frac 1 N \sum_{i=1}^N  G_{\is}^{\mathrm{den}}(\Phi^{-1}(\bm u_i))},\label{eq:ratio}
	\end{equation}
	where $G_{\is}^{\mathrm{den}}(\bm z) = \ell(\bm \mu + \bm L \bm z)W(\bm z)$ and $G_{\is}^{\mathrm{num}}(\bm z)=f(\bm \mu + \bm L \bm z)G_{\is}^{\mathrm{den}}(\bm z)$. 	 Let
	\begin{equation}
		\label{eq:F}
		F(\bm z)=\log \ell(\bm z)=\sum_{i=1}^n Y_i\bm x_i^T \bm z - \sum_{i=1}^n \log(1+\exp(\bm x_i^T \bm z)).
	\end{equation}
	We take 
	$$
		\bm \mu_{\star} = \arg \max_{\bm z \in \mathbb{R}^d} \pi(\bm z)= \arg \max_{\bm z \in \mathbb{R}^d} \ell(\bm z)p(\bm z;\bm 0,\bm I_d),
	$$
	which solves
	$$\bm \mu_{\star} = \nabla F(\bm \mu_{\star})=\sum_{i=1}^n \left(Y_i-\frac{\exp(\bm x_i^T  \bm \mu_{\star})}{1+\exp(\bm x_i^T \bm \mu_{\star})}\right) \bm x_i.$$
	By \eqref{eq:solveoptvar}, the covariance matrix for LapIS is given by
	$$
		\bm \Sigma_{\star} = (\bm I_d - \nabla^2 F(\bm \mu_{\star}))^{-1},
	$$
	where the entries of $\nabla^2 F(\bm \mu_{\star})$ are
	\begin{equation}
		\label{eq:nabla2}
		\nabla^2 F(\bm \mu_{\star})_{jk} = - \sum_{i=1}^n \frac{x_{ij} x_{ik} \exp(\bm x_i^T \bm \mu_{\star})}{(1+\exp(\bm x_i^T \bm \mu_{\star}))^2}.
	\end{equation}
	ODIS and LapIS take $p(\bm z;\bm \mu_{\star},\bm I_d)$ and $p(\bm z;\bm \mu_{\star},\bm \Sigma_{\star})$ as the proposal density, respectively, in estimating the numerator and denominator of \eqref{eq:posterior expectation}. 
	
	\begin{theorem}\label{thm:ratio}
		Assume that the test function $f(\bm z)$ satisfies the boundary growth condition. Applying ODIS with the proposal $p(\bm z;\bm \mu_{\star},\bm I_d)$ for both the numerator and denominator of \eqref{eq:posterior expectation}, then the associated RQMC estimators $\hat{I}_N(G_{\is}^{\mathrm{num}})$ and $\hat{I}_N(G_{\is}^{\mathrm{den}})$ have an RMSE of $O(N^{-1+\epsilon})$. The covariance matrix $\bm \Sigma_{\star}$ of LapIS has at least one eigenvalue less than 1, implying that $W(\bm z;\bm \mu_{\star},\bm \Sigma_{\star})$ fails to satisfy the boundary growth condition.
	\end{theorem}
	\begin{proof}
		It is easy to see that $0 < \ell(\bm z) \leq 1$ and for any $u\subset 1{:}d$, $\abs{(D_u F)(\bm z)}$ is bounded, where $F(\bm z)=\log \ell(\bm z)$ is given by \eqref{eq:F}. Let $A=\{A_1,\dots,A_s\}$ be a split of the set $u$ satisfying $A_i\neq \emptyset$, $\cup_{i=1}^sA_i=u$ and $A_i\cap A_j=\emptyset$ for any $i\neq j$. Let $\mathcal{A}$ be the family of such a set $A$.
		Thus $$\abs{(D_u \ell)(\bm z)}=\abs{\ell(\bm z)\sum_{A\in \mathcal{A}}\prod_{w\in A}(D_{w} F)(\bm z)}\le \sum_{A\in \mathcal{A}}\prod_{w\in A}\abs{(D_{w} F)(\bm z)},$$ which is bounded since the $\abs{(D_{w} F)(\bm z)}$ are bounded. Then $\ell(z)$ satisfies the boundary growth condition. Since $f(\bm z)$ satisfies  the boundary growth condition, $f(\bm z)\ell(\bm z)$ also satisfies the boundary growth condition. For ODIS, $\bm\Sigma =\bm I_d$, by Theorem~\ref{thm:convergencerate}, both $\hat{I}_N(G_{\is}^{\mathrm{num}})$ and $\hat{I}_N(G_{\is}^{\mathrm{den}})$ have an RMSE of $O(N^{-1+\epsilon})$.
		
		For LapIS, we take $\bm \Sigma=\bm \Sigma_{\star}$. If all eigenvalues of $\bm L^T \bm L$ are larger than or equal to 1, we have $$\mathrm{tr}(\bm \Sigma^{-1}) = \mathrm{tr}((\bm L \bm L^T)^{-1}) = \mathrm{tr}((\bm L^T \bm L)^{-1}) \leq d.$$
		On the other hand, by \eqref{eq:nabla2}, we have
		$$\bm \Sigma_{kk}^{-1} = 1+\frac{\sum_{i=1}^n x_{ik}^2 \exp(\bm x_i^T \bm \mu)}{(1+\exp(\bm x_i^T \bm \mu))^2} \geq 1,$$ 
		and thus $\mathrm{tr}(\bm\Sigma^{-1}) > d$ since the data are not all zero.
		This contradiction shows that the matrix $\bm L^T \bm L$ has at least one eigenvalue less than 1. Applying Theorem~\ref{thm:weight} with $\bm\Sigma_0=\bm I_d$, we find that
		the boundary growth condition does not hold for the LR function $W(\bm z;\bm \mu_{\star},\bm \Sigma_{\star})$. Finally, it is easy to see that the covariance matrix $\bm \Sigma_{\star}=\bm L\bm L^T$ of LapIS has the same eigenvalues as $\bm L^T \bm L$.
	\end{proof}
	
	Theorem~\ref{thm:ratio} provides the error bounds for both the numerator and denominator estimators for ODIS. We should remark that the resulting ratio estimator given by \eqref{eq:ratio} is not unbiased. The results in Theorem~\ref{thm:ratio} do not render an error bound for the ratio estimator \eqref{eq:ratio}. Instead, we provide numerical results to illustrate the performance of the ratio estimator. 
	
	We take the Labour Force Participation dataset (used also by \cite{tran2021practical}) in the  numerical study, which contains information of 753 women with one binary variable indicating whether or not they are currently in the labour force together with seven covariates such as number of children under 6 years old, age, education level, etc. The model has $d=8$ unknown parameters  including the intercept. The test function we take is $f(\bm z)=\norm{\bm z}_2^2=\sum_{i=1}^d z_i^2$, which  clearly satisfies the boundary growth condition. For comparison, we present the RMSEs of PriorIS ($\bm \mu =\bm 0,\bm\Sigma =\bm I_d$), ODIS ($\bm \mu =\bm \mu_\star,\bm\Sigma =\bm I_d$), LapIS ($\bm \mu =\bm \mu_\star,\bm\Sigma =\bm \Sigma_{\star}$) in both the MC and RQMC settings. PriorIS takes the prior as the IS density, which can be regarded as plain MC or RQMC. The numerator, denominator and the ratio estimators are all investigated. 
	By Theorem~\ref{thm:ratio}, combining ODIS or PriorIS with RQMC shares an error rate of $O(N^{-1+\epsilon})$ for the numerator and denominator estimators.
	
	We report in Figures~\ref{graph1} and \ref{graph2} the numerical results for the first $30$ and $100$ entries of the Labour Force Participation dataset, respectively.
	In the setting of MC, ODIS and LapIS are more effective than using prior as the proposal (i.e., PriorIS), supporting the benefits of using the two IS methods. In this setting, all methods have RMSEs decaying approximately at the canonical MC rate $O(N^{-1/2})$ as the sample size $N$ increases. 
	
	The situation becomes completely different in the RQMC setting. We observe that ODIS and PriorIS converge faster than LapIS in RQMC. PriorIS even performs better than LapIS in RQMC for the small dataset case as shown in Figure~\ref{graph1}. This is because both the ODIS and PriorIS enjoy the faster error rate $O(N^{-1+\epsilon})$ as predicted by our theoretical analysis. For LapIS, as shown in Theorem~\ref{thm:ratio}, the LR function does not satisfy the boundary growth condition, resulting in an unfavorable integrand for RQMC. 
	
	As the data size is increased from $30$ to 100, Figure~\ref{graph2} shows that LapIS has smaller RMSEs than PriorIS in RQMC, but it still does not improve the MC error rate. LapIS benefits a lot from variance reduction of the integrand as the posterior getting closer to a Gaussian distribution. We also did the simulation for the whole dataset of size 753, and observed that LapIS performs overwhelmingly better than ODIS and PriorIS in both MC and RQMC settings. This is due to the fact that the posterior is very close to a Gaussian distribution for the whole dataset as confirmed by \cite{tran2021practical}. Using LapIS gains a much larger effective sample size, reducing the variance greatly although the convergence rate is not improved. Effective sample sizes for ODIS and PriorIS are very small, resulting in erratic results. We thus omit the results here. In conclusion, using LapIS may be risky in RQMC, particularly when the underlying distribution is far away from Gaussian distributions.

	\begin{figure*}[htbp]
		\begin{center}
			\includegraphics[width=\hsize]{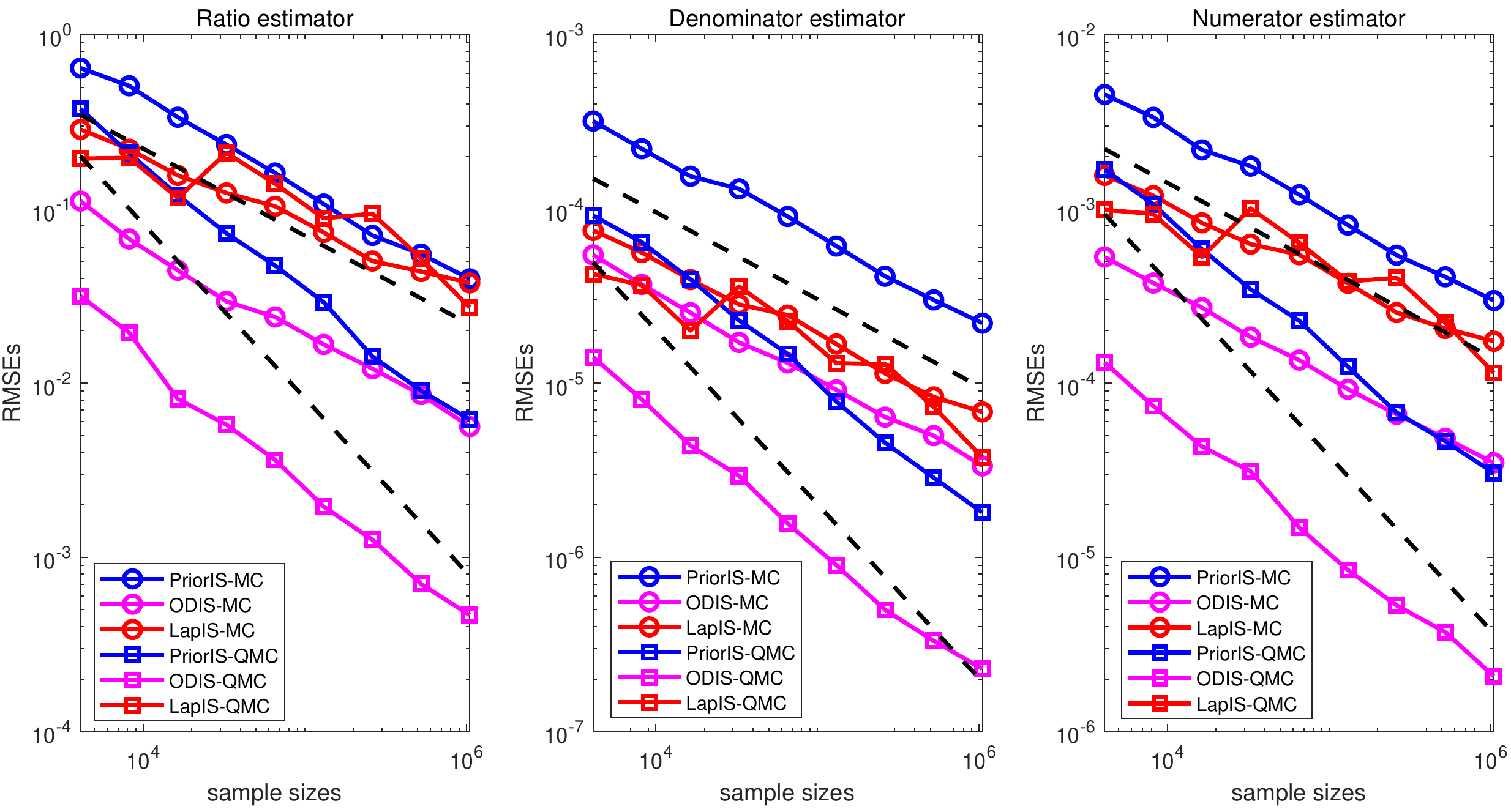}
		\end{center}
		\caption{RMSEs for IS with the multivariate normal distribution. We use the first 30 entries of Labour Force Participation dataset. The dotted lines in the figures represent two convergence rates, $O(N^{-1/2})$ and $O(N^{-1})$ respectively. The RMSEs are computed based on 100 repetitions.}
		\label{graph1}
	\end{figure*}
	
	\begin{figure*}[htbp]
		\begin{center}
			\includegraphics[width=\hsize]{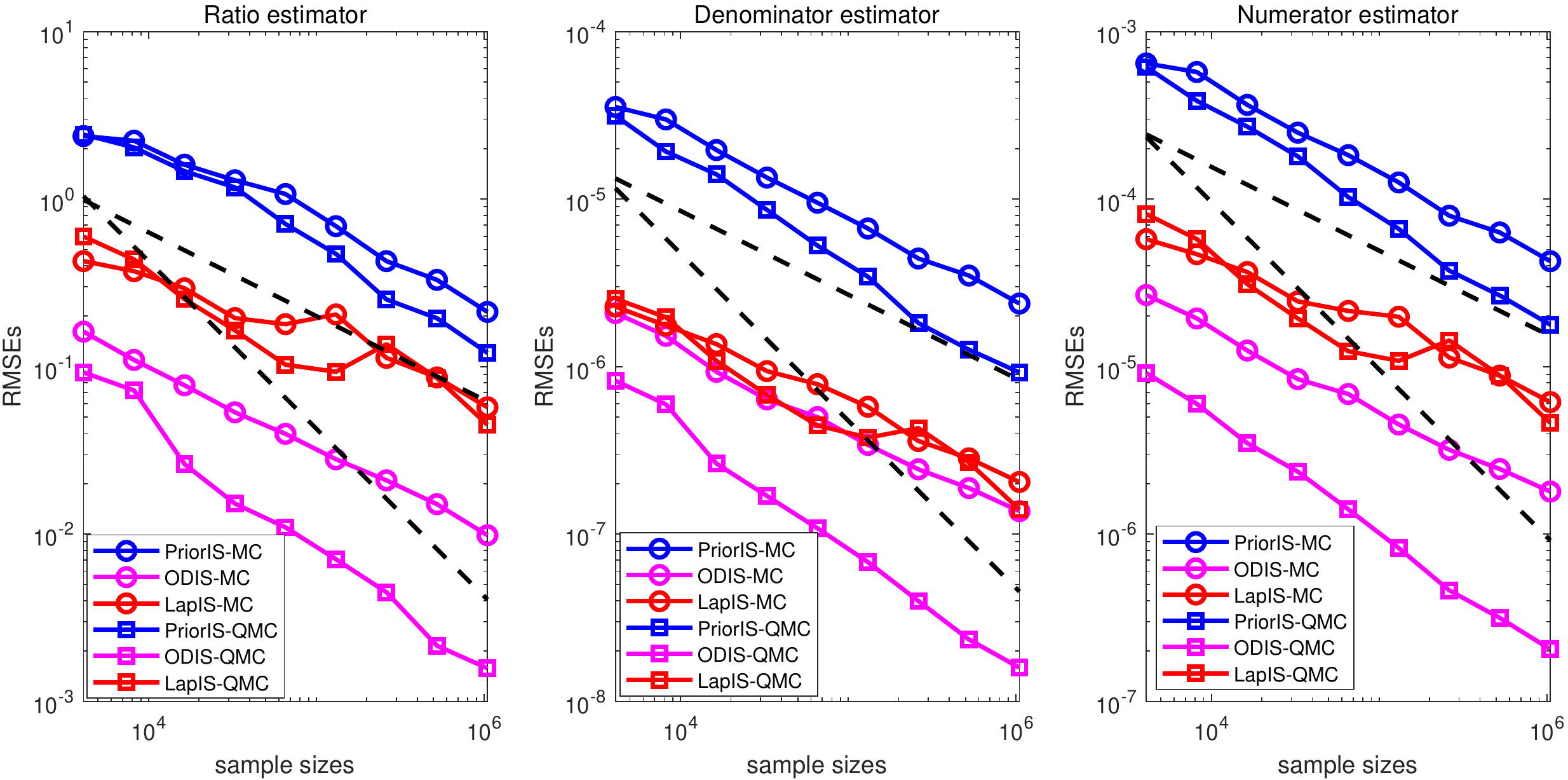}
		\end{center}
		\caption{RMSEs for IS with the multivariate normal distribution. We use the first 100 entries of Labour Force Participation dataset. The dotted lines in the figures represent two convergence rates, $O(N^{-1/2})$ and $O(N^{-1})$ respectively. The RMSEs are computed based on 100 repetitions.}
		\label{graph2}
	\end{figure*}	
	
	Owen \cite{owen2013monte} suggested to use IS with the multivariate $t$ distribution $t_{\nu}(\bm\mu_\star,\bm \Sigma_\star)$ other than $N(\bm\mu_\star,\bm \Sigma_\star)$ for the Bayesian Logistic regression model. Any $\nu$ will lead to an IS distribution with heavier tails than the posterior distribution. We refer to \cite[Chapter 9.7]{owen2013monte} for details. For comparison, we consider three cases of $t_{\nu}(\bm\mu,\bm \Sigma)$, i.e., PriorIS ($\bm \mu =\bm 0,\bm\Sigma =\bm I_d$), ODIS ($\bm \mu =\bm \mu_\star,\bm\Sigma =\bm I_d$), LapIS ($\bm \mu =\bm \mu_\star,\bm\Sigma =\bm \Sigma_{\star}$) in both the MC and RQMC settings. We take a small $\nu=4$.  Figures~\ref{graph11} and  \ref{graph22} show the results for multivariate $t$ distributions as the proposal of IS. Differently from the multivariate normal distributions, RQMC methods yield a similar rate of convergence. LapIS and ODIS in RQMC are comparable, beating PriorIS. This suggests that using IS with  $t_{\nu}(\bm\mu_\star,\bm \Sigma_\star)$ is more robust than that with $N(\bm\mu_\star,\bm \Sigma_\star)$ in QMC.
	
		\begin{figure*}[htbp]
		\begin{center}
			\includegraphics[width=\hsize]{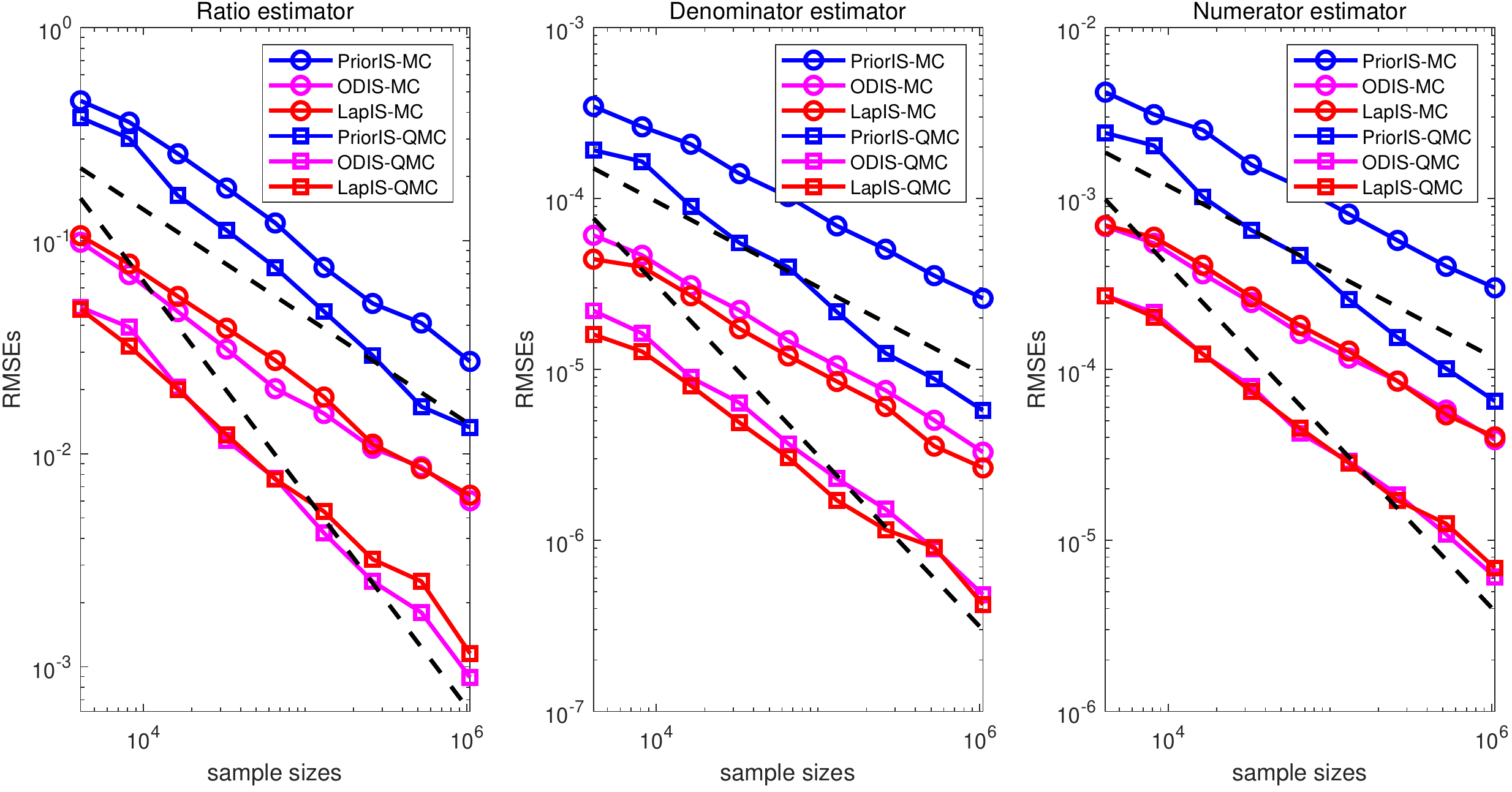}
		\end{center}
		\caption{RMSEs for IS with the multivariate $t$ distribution. We use the first 30 entries of Labour Force Participation dataset. The dotted lines in the figures represent two convergence rates, $O(N^{-1/2})$ and $O(N^{-1})$ respectively. The RMSEs are computed based on 100 repetitions.}
		\label{graph11}
	\end{figure*}
	
	\begin{figure*}[htbp]
		\begin{center}
			\includegraphics[width=\hsize]{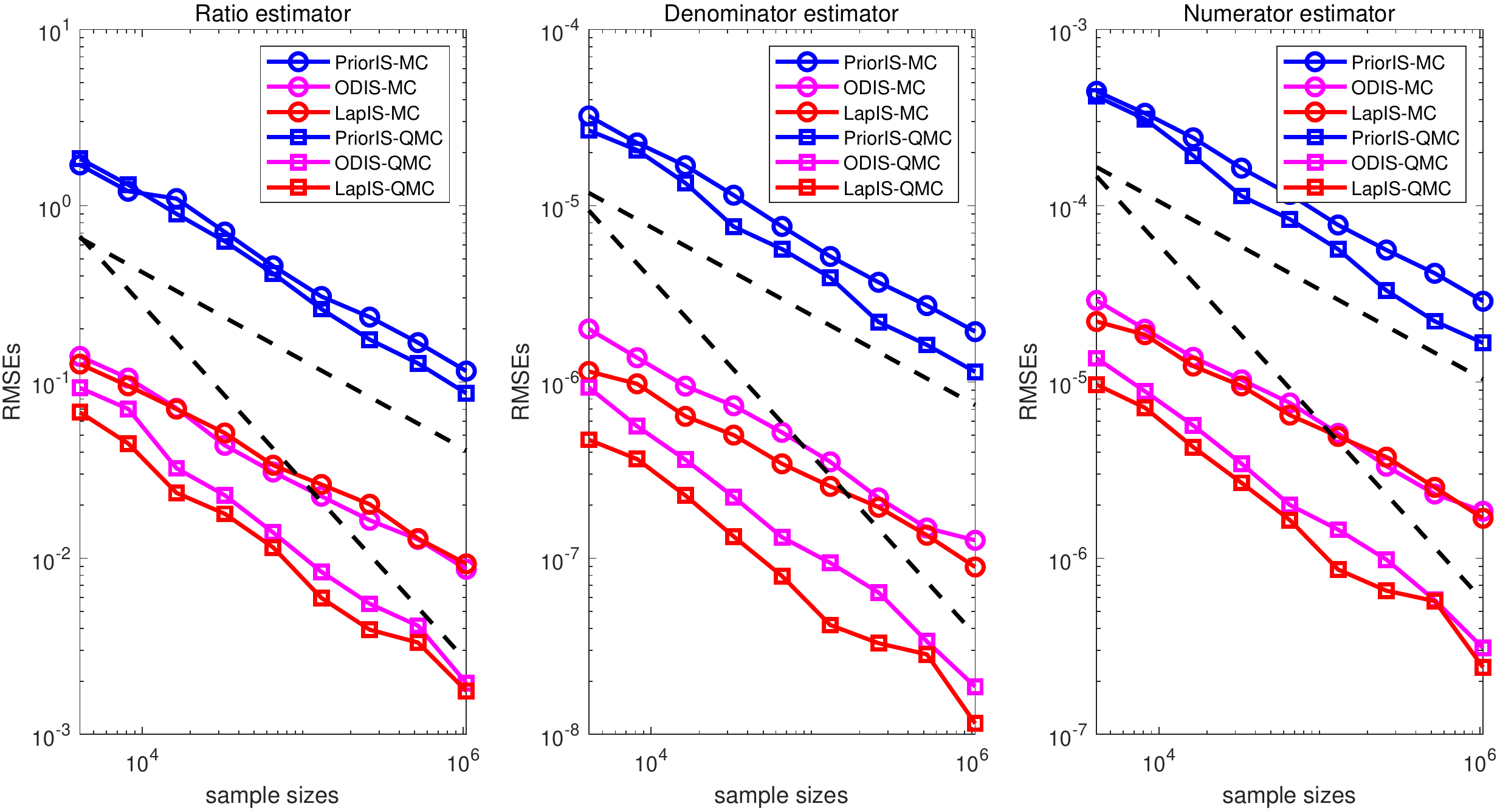}
		\end{center}
		\caption{RMSEs for IS with the multivariate $t$ distribution. We use the first 100 entries of Labour Force Participation dataset. The dotted lines in the figures represent two convergence rates, $O(N^{-1/2})$ and $O(N^{-1})$ respectively. The RMSEs are computed based on 100 repetitions.}
		\label{graph22}
	\end{figure*}	
	
	\section{Conclusion}\label{sec:conc}
	Importance sampling is a classic variance reduction technique widely used in many areas, including finance, rare event simulation, and Bayesian inference. We started from the viewpoint of the boundary growth condition, which is preferable for QMC integration of unbound integrands. We provided sufficient conditions ensuring the boundary growth condition for IS estimators so that nearly $O(N^{-1})$ error rate can be achieved in RQMC with Gaussian or $t$ proposals . We also found that LapIS, a popular IS in practice, does not perform well in the RQMC setting with Gaussian proposals. Generally, an efficient IS in MC is not necessarily efficient in QMC, and vice visa. The ways to assess the performance of IS for MC and QMC are completely different. In addition, when we use $t$ distributions as the proposals, LapIS immediately make a comeback. It is crucial to ask what is a good IS in QMC, beyond the two commonly used LapIS and ODIS. How to choose an appropriate proposal? How to design an IS with a faster convergence rate in QMC, beating LapIS and ODIS? 
	We leave these questions for future research.

	\section*{Appendix}
	In this appendix, we are going to prove Theorem~\ref{thm:mse}.
	Denote $\mu(f) := \int_{(0,1)^d}f(\bm u)d \bm u$, and $\hat\mu_N(f) := \frac{1}{N}\sum_{i=1}^Nf(\bm u_i)$, where $\bm u_1,\dots,\bm u_N$ are a scrambled $(t,m,d)$-net in base $b\ge 2$ with $N=b^m$. Below assume that $f(\bm u)$ satisfies the condition \eqref{eq:grow} in Theorem~\ref{thm:mse}. To avoid the singularities, \cite{owen:2006} used a region as
	\begin{equation*}
		K(\eta) = \left\lbrace\bm u\in[0,1]^d|\prod_{1\le i\le d}\min(u_i,1-u_i)\ge \eta \right\rbrace,
	\end{equation*}
	for small $\eta>0$, and then defined an extension $f_\eta$ of $f$ from $K(\eta)$ to $[0,1]^d$ such that $f_\eta(\bm u)=f(\bm u)$ for $\bm u\in K(\eta)$.  The desired low variation approximation of $f$ is given by
	\begin{equation*}\label{eq:extg}
		f_\eta(\bm u) = f(\bm{c})+ \sum_{v\neq \emptyset}\int_{[\bm c^v,\bm{u}^v]}\partial^{v}f(\bm z^v{:}\bm c^{-v})1\{\bm z^v{:}\bm c^{-v}\in K(\eta)\}d \bm z^v,
	\end{equation*}
	where $\bm c = (1/2,\dots,1/2)$, $\bm z^v{:}\bm c^{-v}$ denotes the point $\bm y\in[0,1]^d$ with $y_j=z_j$ for $j\in v$ and $y_j=c_j$ for  $j\notin v$.
	
	\begin{lem}\label{lem:app}
		If $f$ satisfies the boundary growth condition \eqref{eq:grow}, then for any $\epsilon>0$ there exists  $C_{\epsilon}<\infty$ such that
		\begin{align}
			\mu(\abs{f-f_\eta})&\leq C_{\epsilon}\eta^{1-\max_iB_i-\epsilon},\label{eq:uperror}\\
			\mu((f-f_\eta)^2)&\leq C_{\epsilon}\eta^{1-2\max_iB_i-\epsilon},\label{eq:uperror2}\\
			V_{\mathrm{HK}}(f_\eta)&\leq C_{\epsilon} \eta^{-\max_i B_i-\epsilon}.\label{eq:upHK}
		\end{align}
		If there is a unique maximum among $B_1,\dots,B_d$, then the inequalities hold with $\epsilon=0$.
	\end{lem}
	
	\begin{proof}
		See the proof of Theorem 5.5 in \cite{owen:2006} for establishing \eqref{eq:uperror} and \eqref{eq:upHK}. The inequality \eqref{eq:uperror2} can be proved by the same way of proving the inequality \eqref{eq:uperror}.
	\end{proof}
	
Using triangle inequality gives
	\begin{equation}
		\abs{\hat\mu_N(f)-\mu(f)}\leq \abs{\hat \mu_N(f-f_\eta)}+\abs{\hat\mu_N(f_\eta)-\mu(f_\eta)}+\mu(\abs{f-f_\eta}).\label{eq:trin}
	\end{equation}
	Since each $\bm u_i\sim U{(0,1)^d}$ individually \cite{Owen1995}, we have 
	\begin{equation*}
		\E{\abs{\hat \mu_N(f-f_\eta)}}\leq \frac 1 N \sum_{i=1}^N \E{\abs{f(\bm u_i)-f_\eta(\bm u_i)}}=\mu{(|f-f_\eta|)}.
	\end{equation*}
Using the  Koksma-Hlawka inequality \eqref{eq:hk} with \eqref{eq:upHK}, we find that with probability 1 (w.p.1),
\begin{equation}\label{eq:hkineq}
	\abs{\hat\mu_N(f_\eta)-\mu(f_\eta)}\le V_{\mathrm{HK}}(f_\eta)D^*(\{\bm u_1,\dots,\bm u_N\})=O(\eta^{-\max_i B_i-\epsilon})\times O(N^{-1+\epsilon}),
\end{equation}
for any $\epsilon>0$. In \eqref{eq:hkineq}, we used the fact that scrambled $(t,m,d)$-net is also a $(t,m,d)$-net w.p.1 \cite{Owen1995}, and $D^*(\{\bm u_1,\dots,\bm u_N\})=O(N^{-1}(\log N)^{d})=O(N^{-1+\epsilon})$ for hiding the logarithmic term.
	
By \eqref{eq:uperror}, \eqref{eq:trin} and \eqref{eq:hkineq}, we bound the mean error via
	\begin{align*}
		\E{\abs{\hat\mu_N(f)-\mu(f)}}&\leq 2\mu{(|f-f_\eta|)}+\E{\abs{\hat\mu_N(f_\eta)-\mu(f_\eta)}}\\
		&=O(\eta^{1-\max_iB_i-\epsilon})+O(\eta^{-\max_i B_i-\epsilon})\times O(N^{-1+\epsilon}).
	\end{align*}
Taking $\eta \propto N^{-1}$, the mean error is then of  $\E{\abs{\hat\mu_N(f)-\mu(f)}}=O(N^{-1+\max_i B_i+\epsilon})$ for arbitrarily small $\epsilon>0$, where the arbitrary $\epsilon$ values were adjusted to correspond. This is the main result of \cite{owen:2006}. However, \cite{owen:2006} did not bound the RMSE of $\hat\mu_N(f)$.

To get the analog result for the RMSE, taking the square of \eqref{eq:trin} and the expectation gives
\begin{align*}
	\E{(\hat\mu_N(f)-\mu(f))^2}\leq 3\E{\hat \mu_N(f-f_\eta)^2}+3\E{(\hat\mu_N(f_\eta)-\mu(f_\eta))^2}+3\mu(\abs{f-f_\eta})^2.
\end{align*}
It remains to bound $\E{\hat \mu_N(f-f_\eta)^2}$. We should note that $\bm u_i$ are not independent. By Cauchy-Schwarz inequality and \eqref{eq:uperror2}, we have
\begin{align}
\E{\hat \mu_N(f-f_\eta)^2} &= \E{\left(\frac 1 N\sum_{i=1}^N (f(\bm u_i)-f_\eta(\bm u_i))\right)^2}\notag\\&\le \E{\frac 1 N\sum_{i=1}^N (f(\bm u_i)-f_\eta(\bm u_i))^2}\notag\\&=\E{(f(\bm u)-f_\eta(\bm u))^2}\notag\\&=\mu((f-f_\eta)^2)=O(\eta^{1-2\max_iB_i-\epsilon}).\label{eq:conver}
\end{align}
By using Lemma~\ref{lem:app}, we find that 
\begin{align*}
 	&\E{(\hat\mu_N(f)-\mu(f))^2}\\
 	&=O(\eta^{1-2\max_iB_i-\epsilon})+O(\eta^{-2\max_i B_i-2\epsilon})\times O(N^{-2+2\epsilon})+O(\eta^{2-2\max_iB_i-2\epsilon})\\
 	&=O(\eta^{1-2\max_iB_i-\epsilon})+O(\eta^{-2\max_i B_i-2\epsilon})\times O(N^{-2+2\epsilon}).
\end{align*}
Taking the optimal $\eta\propto N^{-2}$, the RMSE is then $O(N^{-1+2\max_iB_i+\epsilon})$ for arbitrarily small $\epsilon>0$, where the arbitrary $\epsilon$ values were again adjusted to correspond. This suggests that RQMC beats MC when $\max_iB_i<1/4$. Apparently, the RMSE rate is worse than the mean error rate $O(N^{-1+\max_iB_i+\epsilon})$ established before. This is due to the fact that the inequality \eqref{eq:conver} we used is conservative. 

To get an improved upper bound for $\E{\hat \mu_N(f-f_\eta)^2}$, we next make use of a good property of scrambled net quadrature. That is, the scrambled net variance  is no worse than $\Gamma=b^t\left(\frac{b+1}{b-1}\right)^d$ times MC variance \cite{owen1997a}. Based on this result, we find that
\begin{align*}
	\E{\hat \mu_N(f-f_\eta)^2} &= \var{\hat \mu_N(f-f_\eta)}+(\E{\hat \mu_N(f-f_\eta)})^2\\&\le \Gamma\times\frac{ \var{f(\bm u)-f_\eta(\bm u)}}{N}+\mu(|f-f_\eta|)^2\\&\le \Gamma\times\frac{\mu((f-f_\eta)^2)}{N}+\mu(|f-f_\eta|)^2\\
	&=O(\eta^{1-2\max_iB_i-\epsilon}N^{-1}) +O(\eta^{2-2\max_iB_i-2\epsilon}).
\end{align*}
We therefore have
\begin{align*}
	&\E{(\hat\mu_N(f)-\mu(f))^2}\\&=O(\eta^{1-2\max_iB_i-\epsilon}N^{-1})+O(\eta^{2-2\max_iB_i-2\epsilon})+O(\eta^{-2\max_i B_i-2\epsilon})\times O(N^{-2+2\epsilon}).
\end{align*}
Taking the optimal $\eta\propto N^{-1}$, the RMSE is improved to $O(N^{-1+\max_iB_i+\epsilon})$ for arbitrarily small $\epsilon>0$, which completes the proof of Theorem~\ref{thm:mse}.

The arguments above hold also when one uses the first $N$ points of a scrambled $(t,d)$-sequence without requiring the constraint $N=b^m$ on the sample size, but for a different value of $\Gamma<\infty$ (see \cite{gerb:2015} for the details).

\end{document}